\documentclass[11pt]{article}

\usepackage[margin=1.2in]{geometry}

\usepackage{graphicx} 

\usepackage{amsthm,bm}

\usepackage{constants}
\usepackage{enumerate}
\usepackage{dsfont}
\usepackage{booktabs} 
\usepackage{array} 
\usepackage{paralist} 
\usepackage{verbatim} 
\usepackage{subfig} 
\usepackage{caption}
\usepackage{epsfig}
\usepackage{amsmath}
\usepackage{amssymb}
\usepackage{mathrsfs}
\usepackage{mathptmx} 
\usepackage[]{algorithm2e}
\usepackage{hyperref}

\DeclareMathAlphabet{\mathcal}{OMS}{cmsy}{m}{n} 

\usepackage{fancyhdr} 
\usepackage{sectsty}
\allsectionsfont{\sffamily\mdseries\upshape} 

\usepackage{authblk}

\newcommand{\e}{\epsilon}
\newcommand{\beq}{\begin{eqnarray}}

\newcommand{\beas}{\begin{eqnarray*}}
\newcommand{\enas}{\end{eqnarray*}}
\newcommand{\bea}{\begin{eqnarray}}
\newcommand{\ena}{\end{eqnarray}}
\newcommand{\bms}{\begin{multline*}}
\newcommand{\ems}{\end{multline*}}
\newcommand{\bc}{\begin{center}}
\newcommand{\ec}{\end{center}}
\newcommand{\E}{\mathbb{E}}
\newcommand{\A}{\mathcal{A}}
\newcommand{\qmq}[1]{\quad \mbox{#1} \quad}
\newcommand{\qm}[1]{\quad \mbox{#1}}
\newcommand{\s}{\mathbf{s}}
\newcommand{\mz}{\mathbf{z}}
\newcommand{\Bvert}{\left\vert\vphantom{\frac{1}{1}}\right.}

\newcommand{\journal}[1]{} 

\newcommand{\ignore}[1]{}

\newcommand{\U}{{\mathbf U }}

\usepackage{color}

\usepackage{color}

\title{On Strong Embeddings by Stein's Method}

\author{Chinmoy Bhattacharjee and Larry Goldstein}
\affil{Department of Mathematics, University of Southern California}
\date{} 
\begin{document}

\maketitle

\begin{abstract}
Strong embeddings, that is, couplings between a partial sum process of a sequence of random variables and a Brownian motion, have found numerous applications in probability and statistics. We extend Chatterjee's novel use of Stein's method for $\{-1,+1\}$ valued variables to a general class of discrete distributions, and provide $\log n$ rates for the coupling of partial sums of independent variables to a Brownian motion, and results for coupling sums of suitably standardized exchangeable variables to a Brownian bridge.
\end{abstract}

\section{Introduction}

\newtheorem{thm}{Theorem}[section]

Let $\e_1, \e_2 \dots$ be a sequence of independent random variables distributed as $\e$, a mean zero, variance one random variable. Letting $S_k= \sum_{i=1}^k\epsilon_i, k=1,2, \dots$, be the corresponding sequence of partial sums, Donsker's invariance principle \cite{donsker1952justification}, see also \cite{billingsley2013convergence},
implies that the random continuous function
\beas
X_n(t)=\frac{1}{\sqrt n}(S_{[nt]}+(nt-[nt])\e_{[nt]+1}), \quad  0 \le t \le 1
\enas 
converges weakly to a Brownian motion process $(B_t)_{0 \le t \le 1}$. One way to study the quality of the approximation of $X_n(t)$ by $B_t$ is to determine a `slowly increasing' sequence $f(n)$ such that there exists an embedding of both processes on a common probability space such that
	\beas
	\max_{0 \le k\leq n}|S_k-B_k| = O_p(f(n)).
	\enas
Finding the smallest achievable order of $f(n)$ has been a very important question in the literature.

The rate $(n \log \log n)^{1/4}(\log n)^{1/2}$ was achieved by Skorokhod \cite{skorokhod1961issledovaniia}, also see its translation \cite{skhorokhod2014studies} and Strassen \cite{strassen1967} assuming $\E\e^4 < \infty$ using Skorokhod embedding, and Kiefer \cite{kiefer1969deviations} showed that this rate was optimal under the finite fourth moment condition. 
Cs{\"o}rg{\H{o}} and R{\'e}v{\'e}sz \cite{csorgHo1975new} made improvements to the rate under additional moment assumptions. See the survey paper by Ob{\l}{\'o}j \cite{obloj2004skorokhod} and \cite{csorgo2014strong} for a more detailed account.

The celebrated KMT approximation by  Koml{\'o}s, Major and Tusn{\'a}dy (\cite{komlos1975approximation}, \cite{komlos1976approximation}) achieved the rate $\log n$ under the condition that $\e$ have a finite moment generating function in a neighborhood of zero. To state their result precisely we make the following definition.

We say Strong Embedding (SE) holds for the mean zero, variance one random variable $\e$ if there exist constants $C, K$, and $\lambda$ such that for all $n=1,2,\ldots$ the partial sums $S_k= \sum_{i=1}^k\epsilon_i, k=1,\ldots,n$ of a  sequence $\e_1, \e_2 \dots$ of independent random variables distributed as $\e$, and a standard Brownian motion $(B_t)_{t\geq 0}$ can be constructed on a joint probability space such that
\bea \label{eq:goal}
P\left(\max_{0 \le k\leq n}|S_k-B_k|\geq C\log n +x\right)\leq K e^{-\lambda x} \qmq{for all $x \ge 0$.}
\ena
We adopt the standard empty sum convention whereby $S_0=0$.

\begin{thm}[KMT approximation \cite{komlos1975approximation}]\label{KMT}
	SE holds for $\e$ satisfying $\E \exp \theta |\e| < \infty$ for some $\theta >0$.
\end{thm}

Results by B{\'a}rtfai \cite{bartfai1966bestimmung}, see \cite{Zaitsev_estimatesfor}, show that the rate in \eqref{eq:goal} is best possible under the finite moment generating function condition.
	A multidimensional version of the KMT approximation was proved by Einmahl \cite{einmahl1989extensions}, from which Zaitsev (\cite{zaitsev1996estimates}, \cite{zaitsev1998multidimensional}) removed a logarithmic factor. For extensions to stationary sequences see the history in \cite{berkes2014komlos}, where dependent variables of the form $X_k=G(\ldots,\e_{k-1},\e_k,\e_{k+1},\ldots)$ for $\e_i, i \in \mathbb{Z}$ i.i.d. are considered.
	Strong embedding results have a truly extensive range of applications that includes empirical processes, non-parametric statistics, survival analysis, time series, and reliability; for a sampling see the texts \cite{shorack2009empirical} \cite{csorgo2014strong}, or the articles \cite{csorgo1984komlos}, \cite{wu2007inference} and \cite{parzen1979nonparametric}.

Here we take the approach to the KMT approximation introduced by Chatterjee \cite{chatterjee2012new} that has its origins in Stein's method \cite{MR0402873} and
 appears simpler, and is possibly easier to generalize, than the dyadic approximation argument of \cite{komlos1975approximation}. This alternative approach depends on the use of Stein coefficients, also known as Stein kernels, that first appeared in 
the work of Cacoullos and Papathanasiou \cite{cacoullos1992lower}. In some sense, a Stein coefficient $T$ for a mean zero random variable $W$ neatly encodes all information regarding the closeness of $W$ to the mean zero normal variable $Z$ having variance $\sigma^2$. Theorem \ref{thm1} below, from  \cite{chatterjee2012new}, demonstrates that a coupling of $W$ and $Z$ exists whose quality can be evaluated uniquely as a function of $T$ and $\sigma^2$. Theorem \ref{thm2}, that demonstrates Theorem \ref{KMT} for the special case of simple symmetric random walk, was proved in \cite{chatterjee2012new} applying this approach.

\begin{thm}[Chatterjee \cite{chatterjee2012new}] \label{thm2}
	SE holds for $\epsilon$ a symmetric random variable with support $\{-1,+1\}$.
\end{thm}

In this work, using the methods of  \cite{chatterjee2012new},
we generalize Theorem \ref{thm2} as follows. 
\begin{thm} \label{thm3}
	SE holds for $\epsilon$, any random variable with mean zero and variance 1 satisfying $\E\epsilon^3=0$, taking values in a finite set $\A$ not containing $0$.
\end{thm}

 To prove our result we first provide a construction in the case where we have a finite number of variables and then extend to derive strong approximation for an infinite sequence. Such extensions have been studied in the context of the KMT theorem for summands with finite $p$-th moment in \cite{lifshits2000lecture} and also in \cite{chatterjee2012new}.

For the finite case we employ induction, as in \cite{chatterjee2012new}. The induction step requires extending Theorem 1.4 of  \cite{chatterjee2012new} from the special case where $\e$ is a symmetric variable taking values in $\{-1,1\}$. The generalization depends on the `zero-bias' smoothing method introduced in Lemma \ref{lem:smoothing}, which may be of independent interest as regards the construction of Stein coefficients. Theorem \ref{exchangeable} here is  a new result for the embedding of exchangeable random variables and a Brownian bridge.

\begin{thm}\label{exchangeable}
	For any positive integer $n$, let $\e_1, \e_2, \dots, \e_n$ be exchangeable random variables taking values in a finite set ${\cal A} \subset \mathbb{R}$. Let 
	\beas
	S_k=\sum\limits_{i=1}^k \e_i, \quad W_k=S_k-\frac{k}{n}S_n \qmq{and}  \gamma^2=\frac{1}{n}\sum\limits_{i=1}^n \e_i^2.
	\enas
	Then there exists a positive universal constant $C$, and for all $\nu>0$
	positive constants $K_{1},K_{2}$ and $\lambda_0$ depending only on $\A$ and $\nu$, such that for all $n \ge 1$ and $\eta \ge \nu$, a version of  $W_0, W_1, \dots, W_n$ and a standard Brownian bridge ${(B_t)}_{0\le t \le 1}$ exist on the same probability space and satisfy
	\begin{multline*}
	\E \exp(\lambda \max_{0\le k \le n}|W_k - \sqrt{n} \eta B_{k/n} |) \\
	\le \exp(C \log n) \E \exp \left( \frac{K_{1} \lambda^2 S_n^2}{n} + K_{2}\lambda^2 n (\gamma^2-\eta^2)^2\right) \,\,\mbox{for all $\lambda \le \lambda_0$.}
	\end{multline*}
	 Moreover, if $0 \not \in \A$, then there exist positive constants $K_1$ and $\lambda_0$ depending only on ${\cal A}$ such that
	\beas
	\E \exp(\lambda \max_{0\le k \le n}|W_k - \sqrt{n} \gamma  B_{k/n} |) \le \exp(C \log n) \E \exp \left( \frac{K_1 \lambda^2 S_n^2}{n}\right) \,\,\mbox{for all $\lambda \le \lambda_0$,}
	\enas
 and if in addition $\e_1,\ldots,\e_n$ are i.i.d. with zero mean, then there exists a positive $\lambda$ depending only on ${\cal A}$ such that
\beas
P\left(\max_{0\le k \le n}|W_k - \sqrt{n} \gamma  B_{k/n} | \ge \lambda^{-1} C \log n + x\right)  \le 2e^{-\lambda x} \qmq{for all $x \ge 0$.}
\enas 
\end{thm}

\noindent The constant $C$ is given explicitly in \eqref{eq21} in the proof of Theorem \ref{thm4.1'}; its numerical value is roughly 8.4. The constants in the second inequality of Theorem \ref{exchangeable} are those that appear in the first inequality, specialized to a case where the lower bound $\nu$ depends only on ${\cal A}$. 

Our extension of the Rademacher variable result of \cite{chatterjee2012new} requires a number of non-trivial components. 
	Example 3 of \cite{chatterjee2012new} demonstrates how to smooth Rademacher variables to obtain Stein coefficients, and the author states `we do not know yet how to use Theorem 1.2 to prove
	the KMT theorem in its full generality, because we do not know how to
	generalize the smoothing technique of Example 3.' We address this point by the zero bias 
	method of Lemma \ref{lem:smoothing}, that shows how any mean zero, finite variance random variable may be smoothed to obtain a Stein coefficient.

Additionally, dealing with variables restricted to the set $\{-1,1\}$ avoids another difficulty. In particular, the second inequality of Theorem \ref{exchangeable} shows that the `natural scaling' for the approximating Brownian bridge process depends on the variance parameter $\gamma^2=n^{-1}\sum_{i=1}^n \e_i^2$, which in the case of Rademacher variables is always one. In fact, for such variables, the variance parameter remains the constant one when restricted and suitably scaled to any subset of variables. In contrast, in general when applying induction to piece together a larger path from smaller ones, their respective variance parameters may not match. This effect gives rise to the term $(\gamma^2-\eta^2)^2$ in the exponent of the first inequality of Theorem \ref{exchangeable}, which then needs to be controlled in order for the induction to be completed. In doing so, one gains results on the comparison of the sample paths of a more general classes of exchangeable variables to a Brownian bridge.

The second claim of Theorem \ref{exchangeable} is shown under the assumption $0 \not \in {\cal A}$. This condition becomes critical precisely at \eqref{eq:here.is.the.trouble}, where we require that the smallest absolute value of the elements of ${\cal A}$ is positive, from which one then obtains a lower bound $\nu$ on $\gamma$ when invoking Theorem \ref{thm4.1'}. This same phenomenon occurs in the proof of Lemma \ref{lem5.1'}, on the way to demonstrate Theorem \ref{thm3}.

The remainder of this work is organized as follows. In Section \ref{sec2}, we prove two theorems, one for coupling  sums $S_n$ of i.i.d. random variables, and one for coupling $W_n$ of Theorem \ref{exchangeable}, to Gaussians. We also prove Lemma \ref{lem:smoothing}, which shows how to construct Stein type coefficients using smoothing by zero bias variables. Theorems \ref{thm4.1'} and \ref{exchangeable}, the first result a conditional version of the second, are proved in Section \ref{sec3}, and we prove Lemma \ref{lem5.1'}, implying Theorem \ref{thm3}, in Section \ref{sec4}.

\section{Bounds for couplings to Gaussian variables}\label{sec2} 
In this section we prove Theorems \ref{thm3.1'} and \ref{thm3.2'}, generalizations of Theorems 3.1 and 3.2 of \cite{chatterjee2012new}, and our zero bias smoothing result, Lemma \ref{lem:smoothing}. The first theorem gives bounds on couplings of sums $S_n$ of i.i.d. variables, and the second on coupling of certain exchangeable sums to Gaussian random variables.

\begin{thm}
For every mean zero, variance one bounded random variable $\e$ satisfying $\E(\e^3) = 0$ and $\E(\e^4)<\infty$, there exists 
$\theta_1 >0$
such that for every positive integer $n$ it is possible to construct a version of the sum $S_n=\sum_{i=1}^n \e_i$ of $n$ independent copies of $\e$, and $Z_n\sim \mathcal{N}(0,n)$, on a joint probability space such that
\beas
\E \exp(\theta_1 |S_n - Z_n|) \le 8.
\enas
\label{thm3.1'}
\end{thm}

For convenience, we adopt the convention that a normal random variable with mean $\mu$ and zero variance is identically equal to $\mu.$
\begin{thm}
For $n \ge 1$, let $\epsilon_1, \epsilon_2, \dots \epsilon_n$ be arbitrary elements of a finite set $\A \subset \mathbb{R}$, not necessarily distinct. Let $\gamma^2 =n^{-1} \sum_{i=1}^n \e_i^2$,  let $\pi$ be a uniform random permutation of $\{1,2,\dots ,n \}$, and for each $1\le k \le n$ let
\bea \label{eq:defWk}
S_k=\sum_{i=1}^k \epsilon_{\pi(i)} \qmq{and} W_k=S_k - \frac{kS_n}{n}.
\ena
Then for all $\nu>0$ there exist positive constants $c_1,c_2$ and $\theta_2$ depending only on $\A$ and $\nu$ such that for any integer $n \ge 1$, an integer $k$ such that $|2k-n|\le 1$, and any $\eta \ge \nu$, it is possible to construct a version of $W_k$ and a Gaussian random variable $Z_k$ with mean 0 and variance $k(n-k)/n$ on the same probability space such that for all $\theta \le \theta_2$, 
\beas
\E \exp(\theta |W_k - \eta Z_k|) \leq \exp\left(3+\frac{c_1\theta^2 S_n^2}{n}+ c_2\theta^2n(\gamma^2-\eta^2)^2\right).
\enas
\label{thm3.2'}
\end{thm}

We now define Stein coefficients, the key ingredient upon which our approach depends. Let $W$ be a random variable with $\E[W]= 0$ and finite second moment. We say the random variable $T$ defined on the same probability space is a Stein coefficient for $W$ if
\begin{equation}
\label{eq1}
\E[W f (W )] = \E[T f' ( W )]
\end{equation}
for all Lipschitz functions $f$ and $f'$ any a.e. derivative of $f$, whenever these expectations exist.

\begin{thm}[Chatterjee \cite{chatterjee2012new}]
	Let $W$ be mean zero with finite second moment and suppose that $T$ is a Stein coefficient for $W$ with $|T|$ almost surely bounded by a constant. Then, given any $\sigma^2 > 0$, we can construct a version of $W$ and $Z \sim \mathcal{N}(0,\sigma^2)$ on the same probability space such that 
	\beas
	\E \exp(\theta |W-Z|) \leq 2 \E \exp\left(\frac{2\theta^2 (T - \sigma^2)^2}{\sigma^2}\right)  \qmq{for
		all $\theta \in \mathbb{R}$.}
	\enas 
	\label{thm1}
\end{thm}

To prove Theorems \ref{thm3.1'} and \ref{thm3.2'}, we require the following definitions. Following Section 3.2 of \cite{goldstein2007l1}, see also Proposition 4.2 of \cite{chen2010normal},
for $X$ a random variable with finite, non-zero second moment, we say $X^\Box$ has the $X$-square bias distribution when
\bea \label{def:Xbox}
\E[f(X^\Box)]= \frac{1}{\E X^2} \E[X^2 f(X)]
\ena
for all functions $f$ for which the expectation on the right hand side exists. For a mean zero random variable $X$ with finite, non-zero variance $\sigma ^2$, we say that $X^*$ has the $X$-zero bias distribution when
\begin{equation}
\label{eq2}
\sigma^2 \E[f'(X^*)]= \E[Xf(X)]
\end{equation}
for all Lipschitz functions $f$ and any a.e. derivative $f'$, whenever these expectations exist.  That $X^*$ exists for such random variables, see \cite{goldstein1997stein} and \cite{chen2010normal}.

If $X$ is a mean zero random variable with finite, non-zero variance $\sigma^2$, then for any $g \in C_c$, the collection of continuous functions with compact support, letting $f(x)=\int_0^x g(u)du$, using \eqref{def:Xbox}, we have
\begin{align*}
\sigma^2\E g(UX^\Box)&=\sigma^2 \E f'(UX^\Box)\\
&=\sigma^2 \E \int_0^1 f'(uX^\Box)du\\
&=\sigma^2 \E\left[\frac{f(X^\Box)}{X^\Box}\right]\\
&=\E\left[X^2 \frac{f(X)}{X}\right]\\
&=\E[Xf(X)]
\end{align*}
where $X^\Box$ and $U$ are independent, $U\sim \U[0,1]$ and $X^\Box$ has the $X$-square bias distribution. Thus, using \eqref{eq2}, we have
\beas
\sigma^2\E g(UX^\Box)=\E[Xf(X)]=\sigma^2\E [f'(X^*)]=\sigma^2\E [g(X^*)].
\enas
Since the expectation of $g(X^*)$ and $g(UX^\Box)$ agree for any $g \in C_c$, with $=_d$ denoting distributional equivalence, we obtain
\beas
X^*=_d UX^\Box.
\enas

Smoothing $X$ by adding an independent random variable $Y$ having the $X$-zero bias distribution, we obtain the following result which will be used for constructing Stein coefficients for sums.\\
\newtheorem{lem}[thm]{Lemma}
\begin{lem} \label{lem:smoothing}
If $X$ is a mean zero random variable with finite non-zero variance, and $Y$ is an independent variable with the $X$-zero bias distribution, then
	\beas
		\E[X f(X + Y)]= \E[(X^2 - XY) f'(X + Y)]
	\enas
	for all Lipschitz functions $f$ and a.e. derivative $f'$ for which these expectations exist.
	\label{lem3.3'}
\end{lem}
\begin{proof}
	Let $V$ be distributed as $X$, let $U$ be a $\U[0,1]$ random variable, and set
	\beas
		Y=UV^\Box
	\enas
	where $V, U, V^\Box$ and $X$ are independent. Note that for any bivariate function $g$ for which the expectations below exist, by \eqref{def:Xbox} we have
	\begin{equation}
	\label{eq5}
	\E[g(X, V^\Box)]=\frac{1}{\sigma^2}\E[V^2 g(X, V)],
	\end{equation}
	where $\sigma^2$ is the variance of $X$. Hence
	\begin{align*}
	 \E&[(X^2-XY)f'(X+Y)]\\
	=& \E[(X^2-XUV^\Box)f'(X+UV^\Box)]\\
	=& \E\left[\int_{0}^{1}(X^2-XuV^\Box)f'(X+uV^\Box) du\right]\\
	=& \E\left[\frac{(X^2-XuV^\Box)f(X+uV^\Box)}{V^\Box}\Bvert_{0}^{1} + XV^\Box\int_{0}^{1} \frac{f(X+uV^\Box)}{V^\Box} du\right]\\
	=& \E\left[\frac{(X^2-XV^\Box)f(X+V^\Box)-X^2f(X)}{V^\Box}\right] +\E[Xf(X+Y)]\\
	=& \frac{1}{\sigma^2}\E[V(X^2-XV)f(X+V)-VX^2f(X)]+\E[Xf(X+Y)]\\
	=& \frac{1}{\sigma^2}\E[VX(X-V)f(X+V)]+\E[Xf(X+Y)],
	\end{align*}
	where we have used $\eqref{eq5}$ in the second to last equality, as well as the independence of $V$ and $X$, and that $\E V = 0$, in the last. Hence, to prove the claim, it suffices to show that the first term above is zero. Since $X =_d V$ and $V$ and $X$ are independent and exchangeable, we have
	\begin{equation*}
	VX(X-V)f(X+V)=_d VX(V-X)f(X+V)=-VX(X-V)f(X+V),
	\end{equation*}
	demonstrating that the expectation of the expression above is zero.
\end{proof}

For any mean zero $X$ with finite, non-zero variance $\sigma^2$ the distribution of $X^*$ is absolutely continuous with density function
\begin{equation}
\label{eq3}
p_{X^*}(x) = \frac{\E[X \mathds{1}(X > x)] }{\sigma^2}.
\end{equation}
One finds directly from \eqref{eq3} that
\bea \label{Xbd.implies.X*bd}
\mbox{$a \le X \le b$ for some constants $a<b$ implies $a \le X^* \le b$.}
\ena

Comparing \eqref{eq1} with \eqref{eq2}, we see that $T$ is a Stein coefficient for $X$ if $\sigma^{-2}E[T|X]$ is the Radon Nikodym derivative $\frac{d\mu^*}{d\mu}$ of the probability measure $\mu^*$ of $X^*$ with respect to the measure $\mu$ of $X$. Hence, in light of \eqref{eq3}, if
$X$ is a random variable with mean zero and finite variance, having density function $p_{X}(x)$ whose support is an interval, then setting
\begin{equation}
\label{eq4}
h_{X}(x)=\frac{\E[X\mathds{1}(X>x)]}{p_{X}(x)}\mathds{1}(p_{X}(x)>0) \qquad \text{we have} \qquad \E[X f(X)]=\E[h_{X}(X) f'(X)]
\end{equation}
for all Lipschitz function $f$ and a.e. derivative $f'$ for which these expectations exist, that is, $h_{X}(X)$ is a Stein coefficient for $X$. We note the first equality in \eqref{eq4} shows, by virtue of $\E(X)=0$, that $h_{X}(x) \ge 0$.

Now consider a random variable $X$ having vanishing first and third moment, variance strictly between zero and infinity and satisfying $\E(X^4)<\infty$. Then the distribution for a random variable $Y$ having the $X$-zero bias distribution exists, and from \eqref{eq2} with $g(x)=x^2$ and $g(x)=x^3$, we respectively find
\bea \label{eq:Ymean.zero.finite.var}
\E(Y)=0 \qmq{and} \E(Y^2)<\infty.
\ena
Moreover from \eqref{eq3} we see that $Y$ has density function $p_Y(y)$ whose support is a closed interval. Hence the function $h_Y(y)$, given by the first equality in $\eqref{eq4}$, satisfies the second.

\begin{lem} \label{lem:tilde.smoothing}
Let $\e_1, \dots \e_n$ be independent and identically distributed as $\e$, a random variable with mean zero, finite nonzero variance, and satisfying $\E(\e^3) = 0 $ and $\E(\e^4)<\infty$, and let $Y$ have the $\e$-zero bias distribution and be independent of $ \e_1, \dots, \e_n.$ Then for all Lipschitz functions $f$ and a.e. derivative $f'$,
\beas 
\E[\widetilde S_n f(\widetilde S_n)]=\E[T f'(\widetilde S_n)]
\enas
where
\beas
\widetilde S_n= S_n + Y \qmq{with} S_n=\e_1 + \e_2 + \dots + \e_n,
\enas
and
\beas
T=\sum\limits_{i=1}^{n}\e_i^2 - S_n Y+h_Y(Y)  \qmq{with}  h_Y(y)=\frac{\E[Y\mathds{1}(Y>y)]}{p_Y(y)}\mathds{1}(p_Y(y)>0).
\enas
\end{lem}
\begin{proof}
With $S_n^{(i)}=S_n - \e_i$, we have
\bea 
\label{eq7}
\E[\widetilde S_n f(\widetilde S_n)]=\E[S_n f(\widetilde S_n) +Y f(\widetilde S_n)]
=\sum\limits_{i=1}^{n}\E[\e_i f(\e_i + Y+ S_n^{(i)})] + \E[Yf(Y + S_n)].
\ena

For the first term of \eqref{eq7}, using that the summands $\e_i$ are independent and applying Lemma $\ref{lem3.3'}$
yields
\beas
\E[\e_i f(\e_i + Y + S_n^{(i)})]=\E[(\e_i^2 - \e_i Y)f'(\e_i +Y + S_n^{(i)})]=\E[(\e_i^2 - \e_i Y)f'(\widetilde S_n)].
\enas
Now turning to the second term of \eqref{eq7}, we first note that by \eqref{eq2} the assumption that the third moment of $\e$ is zero implies $E(Y)=0$. Now using the independence of $Y$ and $S_n$, \eqref{eq4} yields
\beas
\E[Y f(Y + S_n)] = \E[h_Y(Y)f'(Y + S_n)] = \E[h_Y(Y)f'(\widetilde S_n)].
\enas
Substitution into \eqref{eq7} now yields the claim.
\end{proof}

Hoeffding's lemma, e.g. see the proof of Lemma 2.2 of \cite{boucheron2013concentration}, will be used below. It states that if $X$ is a mean zero random variable that satisfies $a \le X \le b$ almost surely, then
	\bea \label{Hoeffding.lemma}
	\E[\exp(\theta X)] \le e^{(b-a)^2 \theta^2/8} \qmq{for all $\theta \in \mathbb{R}$.}
	\ena 
We also require the `non central $\chi_1^2$' moment generating function identity,
\bea \label{chi}
{\E \exp\left( \alpha V^2 + \beta V \right) = \frac{\exp\left(\frac{\beta^2}{2(1-2\alpha)}\right)}{(1-2 \alpha )^{1/2}}}
\ena
valid for the standard Gaussian variable $V$, and all $\beta \in \mathbb{R}$ and $\alpha<1/2$. 

For the law ${\cal L}(X)$ of any random variable $X$ let 
\beas
\ell({\cal L}(X))=\inf\{b-a: P(a \le X \le b)=1\},
\enas
the length of the support of $X$. For notational simplicity we will write $\ell(X)$, or $\ell$ when $X$ is clear from context, for $\ell({\cal L}(X))$. We use that $\ell(X)$ is translation invariant in the sense that $\ell(X)=\ell(X-c)$ for any real number $c$ without further mention.

\begin{lem}
\label{hoeff}
For every almost surely bounded random variable $X$, there exists a constant $\vartheta_{\ell(X)} \in (0,\infty)$ depending only on $\ell(X)$ such that when $X_1, X_2,\dots$ are independent random variables distributed as $X$, the sum $S_n=X_1+\cdots+X_n$ and $\mu=\E X$ satisfy
\beas
\E\left[\exp\left(\theta^2\frac{S_n^2}{n}\right)\right] \le \frac{4}{3}\exp\left(\frac{4}{3}n \theta^2 \mu^2\right) \qmq{for all $n \ge 1$ and $|\theta| \le \vartheta_{\ell(X)}$.}
\enas
\end{lem}

The constant $4/3$ is somewhat arbitrary as any value greater than 1 can be achieved; the proof of Theorem \ref{thm4.1'} requires a value strictly less than $3/2$.
\begin{proof}
Let $V$ be a $\mathcal{N}(0,1)$ random variable independent of $X$. 
Using Hoeffding's lemma \eqref{Hoeffding.lemma} conditional on $V$, for any function of $V$ we have
\beas
\E[\exp(t(V)(X-\mu))|V] \le e^{\ell^2 t(V)^2/8}.
\enas
Applying $\E(\exp \theta V)=\exp(\theta^2/2)$, for $\ell \theta < \sqrt{2}$ and $V$ independent of $X_1,X_2,\ldots$, letting $t(V)= \sqrt{2} \theta \frac{V}{\sqrt{n}}$ we obtain
\begin{multline*}
\E\left[\exp\left(\theta^2\frac{S_n^2}{n}\right)\right] = \E\left[\exp\left( {\sqrt 2} \theta \frac{S_n}{\sqrt n}V \right)\right]
=\E\left[\E\left(\exp\left({\sqrt 2} \theta \frac{V}{\sqrt n}X\right)\Bvert V\right)^n \right]\\
= \E
\left[ 
\E\left(\exp(t(V)X)\Bvert V \right)^n 
\right]=\E
\left[ 
\E\left(\exp (t(V)(X-\mu)+t(V)\mu) \Bvert V \right)^n 
\right]\\
\le \E\left[\exp\left(\frac{2{\ell}^2\theta^2 V^2}{8n}+\sqrt{2}\theta \mu \frac{V}{\sqrt n} \right)^{n}\right] = \E \left[\exp \left (\frac{{\ell}^2 \theta^2}{4} V^2 + \sqrt{2}\theta \mu \sqrt n V\right) \right] \\
= \frac{1}{\sqrt{1- \ell^2 \theta^2/2}}\exp \left(\frac{n \theta^2 \mu^2}{1- {\ell}^2 \theta^2/2} \right)\le \frac{1}{1- \ell^2 \theta^2/2}\exp \left(\frac{n \theta^2 \mu^2}{1- {\ell}^2 \theta^2/2} \right),
\end{multline*}
where we have applied \eqref{chi} in the last line. It is now direct to verify that the property required by the lemma holds by letting $\vartheta_{{\ell(X)}}=1/(\sqrt{2}\ell(X))$, the unique positive solution to
\beas 
\frac{1}{1- {\ell(X)}^2 \theta^2/2}=\frac{4}{3}.
\enas 
\end{proof}

\begin{lem}
\label{h_Y}
Let $\e$ be a bounded, mean zero, variance $\sigma^2\in(0,\infty)$ random variable satisfying
$\E \e^3=0$.
Then the Stein coefficient $h_Y(y)$, given by \eqref{eq4} for $Y$ with the $\e$-zero bias distribution, is bounded.
\end{lem}
\begin{proof} 
As $\e$ is a mean zero random variable with finite, nonzero variance, the zero bias distribution ${\cal L}(Y)$ exists. As $\E \e^3=0$ and $\e$ is bounded and non-trivial, as in \eqref{eq:Ymean.zero.finite.var} one verifies that $\E Y =0$ and that ${\rm Var}(Y)$ is positive and finite. Hence, as noted below \eqref{eq4}, the Stein coefficient $h_Y(y)$ as given by \eqref{eq4} is nonnegative, so we need only show that it is bounded above.

From \eqref{eq3}, an a.e. density of $Y$ is given by
\begin{equation}
\label{p_Y}
p_Y(y)=\frac{1}{\sigma^2}\int_y^\infty udF_{\e}(u)
\end{equation}
where we use $F_X$ to denote the distribution function of the random variable $X$. From \eqref{p_Y} we may observe that the support of $Y$ is the smallest closed interval of $\mathbb{R}$ containing the support of $\e$.
Since $\e$ is bounded and has mean zero, using \eqref{Xbd.implies.X*bd}, this interval is of the form $[a,b]$ for $-\infty<a<0<b<\infty$, hence for $t \in [a,b]$ the upper limit of the integral in \eqref{p_Y} may be replaced by $b$.

In particular, for all $t\in [0,b]$, by \eqref{eq4} we have
\begin{equation*}
\begin{split}
h_Y(t)&=\frac{\int_{t}^b yp_Y(y) dy}{p_Y(t)}\\
& = \frac{\int_{t}^b y  \int_y^b udF_{\e}(u) dy}{\sigma^2 p_Y(t)} \\
& = \frac{\int \int_{t \le y \le u \le b} y  udF_{\e}(u) dy}{\sigma^2 p_Y(t)}\\
& = \frac{\int_t^b u \int_t^u y  dy dF_{\e}(u) }{\sigma^2 p_Y(t)}\\
& = \frac{\int_t^b u (u^2-t^2) dF_{\e}(u) }{2\sigma^2 p_Y(t)}\\
& \le \frac{b^2 \int_t^b udF_{\e}(u) }{2\sigma^2 p_Y(t)} = \frac{b^2}{2},
\end{split}
\end{equation*}
where we have used Fubini's theorem in the fourth equality, and \eqref{p_Y} in the second and sixth. As $h_{-Y}(t)=h_Y(-t)$ we obtain that $h_Y(y)$ is bounded for $t \in [a,0]$.
\end{proof}

\noindent \textit{Proof of Theorem $\ref{thm3.1'}$:}
For short we write $S=\e_1+\e_2+ \dots + \e_n$ and $\widetilde S=S + Y$ with $Y$ is as in Lemma \ref{lem:tilde.smoothing}. As the third moment of $\e$ is zero and its fourth moment is finite, as in \eqref{eq:Ymean.zero.finite.var}, $Y$ has mean zero with finite variance, and hence so does $\widetilde S$.

By Lemma \ref{lem:tilde.smoothing}, $T=\sum\limits_{i=1}^{n}\e_i^2 - S Y+h_Y(Y)$ is a Stein coefficient for $\widetilde S$. 
Since $\e$ is bounded and the third moment of $\e$ is zero, Lemma \ref{h_Y} yields that $h_Y(Y)$ is bounded. Also $\e$ bounded implies $S$ is bounded. In addition, as $\e$ is bounded there exists some $B$ such that $|\e| \le B$, and \eqref{Xbd.implies.X*bd} implies $|Y| \le B$. Thus, we conclude that $|T|$ is bounded.

Now invoking Theorem \ref{thm1}, there exists a version of $\widetilde S$ and $Z \sim \mathcal{N}(0, \sigma^2)$ on the same probability space such that
\beas
\E \exp(\theta|\widetilde S-Z|)\le 2\E\left(\exp(2\theta^2\sigma^{-2}(T-\sigma^2)^2)\right) \qmq{for all $\theta \in \mathbb{R}$.}
\enas
Using $|Y| \le B$ we have $|S-\widetilde S|\le B$. It follows that,
\beas
\E \exp(\theta|S-Z|)\le 2\E\left(\exp(B|\theta| + 2\theta^2\sigma^{-2}(T-\sigma^2)^2)\right).
\enas

Letting $C_0 \ge B$ be such that $|h_Y(Y)|\le C_0$, and setting
$\sigma^2=n$, we obtain
\begin{equation*}
\frac{(T-\sigma^2)^2}{\sigma^2}\le \frac{3{\overline S}^2+3C_0^2 S^2+3C_0^2}{n}
\end{equation*}
where ${\overline S}=\sum\limits_{i=1}^{n} (\e_i^2-1)$. \\

Hence,
\begin{align}\label{eq:exp.convex}
\E \exp(\theta |S-Z|)  &\le 2\exp\left(B| \theta |+\frac{6C_0^2 \theta^2}{n} \right)\E \exp\left(6\theta^2\frac{{\overline S}^2 +C_0^2 S^2}{n}\right)\nonumber\\ 
 &\le \exp\left(B|\theta |+\frac{6C_0^2 \theta^2}{n} \right)\E\left[\exp\left(\frac{12\theta^2{\overline S}^2}{n}\right) +\exp\left(\frac{12\theta^2C_0^2 S^2}{n}\right)\right]
\end{align}
where we applied the simple inequality $\exp(x+y)\le (e^{2x}+e^{2y})/2$.\\\\
Noting for $\overline S$ and $S$ that $\e^2-1$ and $\e$ respectively are bounded and have mean zero, using Lemma \ref{hoeff} for the first two inequalities below, we see that there exists $\theta_1>0$ such that for all $|\theta| \le \theta_1$ and all positive integers $n$
\begin{equation*}
\E \exp(12\theta^2{\overline S}^2 /n)
\le 2\quad \text{and} \quad \E \exp(12\theta^2C_0^2 S^2/n)
\le 2 \qmq{and} \exp\left(B|\theta |+\frac{6C_0^2 \theta^2}{n} \right)\le 2.
\end{equation*}
Theorem \ref{thm3.1'} now follows from \eqref{eq:exp.convex}.
\qed\\

We now prepare for the proof of Theorem \ref{thm3.2'} by providing a few lemmas. For $\A$ the finite set in which the basic variable $\e$ takes values, let
\bea \label{eq:defcalD}
{\cal D}=\{b-a: a,b \in \A \} \qmq{and} {\cal D}^+ = {\cal D} \cap [0,\infty),
\ena
the set of differences of the elements in ${\A}$, and those differences that are non-negative. We note here that ${\cal D}$ is symmetric in that ${\cal D}=-{\cal D}$. Let also
\bea \label{def:B.is.max}
B=\max_{a \in {\cal A}}|a|.
\ena 

Recall the definition \eqref{eq:defWk} of $W_k$ and
observe that we may write
\begin{multline}
W_k=S_k-\frac{k}{n}S_n = \sum_{i=1}^k \e_{\pi(i)}-\frac{k}{n}\sum_{i=1}^n\e_{\pi(i)}=\frac{n-k}{n}\sum_{i=1}^k \e_{\pi(i)}-\frac{k}{n}\sum_{i=k+1}^n e_{\pi(i)}\\=\frac{1}{n}\sum_{i=1}^k \sum_{j=k+1}^n(\e_{\pi(i)}-\e_{\pi(j)}), \label{eq:WkDblSum}
\end{multline}
and therefore
\bea \label{eq:defWkd}
W_k=\sum\limits_{d\in {\cal D}^+} W_{k,d} \qmq{where}
W_{k,d}=\frac{1}{n}\sum\limits_{i=1}^{k}\sum\limits_{j=k+1}^{n}( \e_{\pi(i)} - \e_{\pi(j)})
 \mathds{1}_{(|\e_{\pi(i)} - \e_{\pi(j)}|=d)}.
\ena

\begin{lem}
Under the hypotheses of Theorem \ref{thm3.2'}, for any $\theta \in \mathbb{R}, 1\le k\le n$ and $d \in {\cal D}^+$ we have
\bea \label{eq9}
\E \exp(\theta W_{k,d}/\sqrt k)\le \exp(d^2 \theta^2/2) \qmq{and} \E \exp(\theta W_{k}/\sqrt k)\le \exp(B^2 \theta^2),
\ena
where $B$ is as in \eqref{def:B.is.max}. Further, there exists $\alpha_0>0$ depending only on ${\cal A}$ such that
\bea
\label{bdWkd}
\E [\exp (\alpha W_{k,d}^2/k)]\le 2 \qmq{for all  $|\alpha| \le \alpha_0$ and all $d \in {\cal D^+}$.}
\ena  
\label{lem3.4'}
\end{lem}
\begin{proof}
We may assume $d>0$ as the result is otherwise trivial. Fix an integer $k$ in $[1,n]$ and $d \in {\cal D}^+$, and let $m(\theta) := \E \exp(\theta W_{k,d}/\sqrt k)$.
We argue as in \cite{chatterjee2012new}. Since $W_{k,d}$ is bounded, the function $m(\theta)$ is differentiable and
differentiation and expectation may be interchanged. Hence, using \eqref{eq:defWkd} for the second equality,
\bea
\nonumber
m'(\theta)&=&\frac{1}{\sqrt k}\E(W_{k,d}\exp(\theta W_{k,d}/\sqrt k))\\  \label{eq8}
&=&\frac{1}{n\sqrt k}\sum\limits_{i=1}^{k}\sum\limits_{j=k+1}^{n}\E[(\e_{\pi(i)} - \e_{\pi(j)}) \mathds{1}_{(|\e_{\pi(i)} - \e_{\pi(j)}|=d)} \exp(\theta W_{k,d}/\sqrt k)].
\ena

Now, let $i$ and $j$ satisfying $1 \le i \le k < j \le n$ be arbitrary and let $\pi '=\pi \circ (i,j)$ where $(i,j)$ is the transposition
of $i$ and $j$. Then $(\pi, \pi')$ is an exchangeable
pair of random permutations. Let
$W_{k,d}'$ be defined as in \eqref{eq:defWkd} with $\pi'$ replacing $\pi$.
Using exchangeability for the first equality and the definition of $\pi'$ for the second,
\begin{align*}
\E[&(\e_{\pi(i)} - \e_{\pi(j)}) \mathds{1}_{(|\e_{\pi(i)} - \e_{\pi(j)}|=d)} \exp(\theta W_{k,d}/\sqrt k)]\\
=& \E[(\e_{\pi '(i)} - \e_{\pi '(j)}) \mathds{1}_{(|\e_{\pi'(i)} - \e_{\pi'(j)}|=d)} \exp(\theta W'_{k,d}/\sqrt k)]\\
=& \E[(\e_{\pi(j)} - \e_{\pi(i)}) \mathds{1}_{(|\e_{\pi(i)} - \e_{\pi(j)}|=d)} \exp(\theta W'_{k,d}/\sqrt k)]\\
=& -\E[(\e_{\pi(i)} - \e_{\pi(j)}) \mathds{1}_{(|\e_{\pi(i)} - \e_{\pi(j)}|=d)} \exp(\theta W'_{k,d}/\sqrt k)] .
\end{align*}
Averaging the first and last expressions yields
\begin{align}
&\E[(\e_{\pi(i)} - \e_{\pi(j)}) \mathds{1}_{(|\e_{\pi(i)} - \e_{\pi(j)}|=d)} \exp(\theta W_{k,d}/\sqrt k)]\nonumber\\
\label{eq:ave.pi0}=&\frac{1}{2}\E[(\e_{\pi(i)} - \e_{\pi(j)}) \mathds{1}_{(|\e_{\pi(i)} - \e_{\pi(j)}|=d)}(\exp(\theta W_{k,d}/\sqrt k) - \exp(\theta W'_{k,d}/\sqrt k ))].
\end{align}
Note
\begin{equation*}
\begin{split}
& |W_{k,d} - W'_{k,d}|\\
 = &\frac{1}{n}\Bvert \sum_{k+1 \le l \le n, l \not =j}(\e_{\pi(i)} - \e_{\pi(l)}) \mathds{1}_{(|\e_{\pi(i)} - \e_{\pi(l)}|=d)}+ \sum_{1 \le l \le k, l \not =i}(\e_{\pi(l)} - \e_{\pi(j)}) \mathds{1}_{(|\e_{\pi(l)} - \e_{\pi(j)}|=d)}\\
 & \qquad + (\e_{\pi(i)} - \e_{\pi(j)}) \mathds{1}_{(|\e_{\pi(i)} - \e_{\pi(j)}|=d)} - \Big(\sum_{k+1 \le l \le n, l \not =j}(\e_{\pi'(i)} - \e_{\pi'(l)}) \mathds{1}_{(|\e_{\pi'(i)} - \e_{\pi'(l)}|=d)}\\
 & \qquad + \sum_{1 \le l \le k, l \not =i}(\e_{\pi'(l)} - \e_{\pi'(j)}) \mathds{1}_{(|\e_{\pi'(l)} - \e_{\pi'(j)}|=d)}+ (\e_{\pi'(i)} - \e_{\pi'(j)}) \mathds{1}_{(|\e_{\pi'(i)} - \e_{\pi'(j)}|=d)}\Big) \Bvert\\
= & \frac{1}{n}\Bvert\sum\limits_{l=1}^n (\e_{\pi(i)} - \e_{\pi(l)}) \mathds{1}_{(|\e_{\pi(i)} - \e_{\pi(l)}|=d)} + \sum\limits_{l=1}^n (\e_{\pi(l)} - \e_{\pi(j)}) \mathds{1}_{(|\e_{\pi(l)} - \e_{\pi(j)}|=d)}\Bvert\\ \le &\frac{1}{n}[nd+ nd]=2d.
\end{split}
\end{equation*}

Now applying the inequality $|e^x - e^y|\le \frac{1}{2}|x-y|(e^x + e^y)$ we see that  \eqref{eq:ave.pi0} in absolute value is bounded by
\begin{equation*}
\begin{split}
& \frac{|\theta|}{4\sqrt k}\E[|\e_{\pi(i)} - \e_{\pi(j)}| \mathds{1}_{(|\e_{\pi(i)}-\e_{\pi(j)}|=d)}|W_{k,d} - W'_{k,d}|((\exp(\theta W_{k,d}/\sqrt k) + \exp(\theta W'_{k,d}/\sqrt k ))]\\
 &\le \frac{|\theta|}{4\sqrt k}2d^2 \E[\exp(\theta W_{k,d}/\sqrt k) + \exp(\theta W'_{k,d}/\sqrt k )]\\
 &= \frac{|\theta| d^2}{\sqrt k} m(\theta).
\end{split}
\end{equation*}

So, from \eqref{eq8}, and the fact that $1 \le i \le k$ and $k<j \le n$ are arbitrary, we obtain
\beas
|m'(\theta)|\le \frac{1}{n \sqrt k} \frac{|\theta| d^2}{\sqrt k} \sum\limits_{i=1}^{k}\sum\limits_{j=k+1}^{n} m(\theta) \le d^2 |\theta| m(\theta).
\enas
Now, using, $m(0)=1$, and that $m(\theta)\ge 0$  for all $\theta \in \mathbb{R}$, for $\theta > 0$, we obtain
\beas
\int_0^{\theta} \frac{m'(u)}{m(u)}du \le \int_0^{\theta} d^2 u du \implies m(\theta)\le \exp(d^2 \theta^2/2)
\enas
and for $\theta < 0$, we obtain
\beas
\int_{\theta}^0 \frac{-m'(u)}{m(u)}du \le \int_{\theta}^0 d^2(-u) du \implies m(\theta)\le \exp(d^2 \theta^2/2),
\enas
proving the first inequality of \eqref{eq9}.

Arguing similarly, now letting $m(\theta) := \E \exp(\theta W_k/\sqrt k)$ and $W'_k$ as in \eqref{eq:defWk} with $\pi'$ replacing $\pi$, noting $|W_k-W'_k|=|\e_{\pi(i)}-e_{\pi(j)}|\le 2B$, we obtain
\begin{equation*}
\begin{split}
\E[(\e_{\pi(i)} - \e_{\pi(j)}) \exp(\theta W_{k} /\sqrt k)]& \le \frac{|\theta|}{4\sqrt k}\E[(\e_{\pi(i)} - \e_{\pi(j)})^2 ((\exp(\theta W_{k}/\sqrt k) + \exp(\theta W_{k}'/\sqrt k ))]\\
&  \le \frac{|\theta|4 B^2}{4\sqrt k}2m(\theta)=\frac{2|\theta| B^2}{\sqrt k}m(\theta),
\end{split}
\end{equation*}
so that  $|m'(\theta)|\le 2 B^2 |\theta| m(\theta)$, implying the final inequality of \eqref{eq9}.

Turning to \eqref{bdWkd}, letting $Z$ be a standard normal random variable independent of $W_{k,d}$, by  \eqref{eq9} and \eqref{chi}, for all $d \in {\cal D}^+$ and $\alpha  < 1/(2d^2)$,  we have
\beas
\E \exp(\alpha W_{k,d}^2/k) & = \E \exp\left(\sqrt{2\alpha} Z W_{k,d}/\sqrt{k}  \right) 
\le \E \exp (d^2 \alpha Z^2)
\le \frac{1}{\sqrt {1-2d^2 \alpha}}.
\enas
Now set $\alpha_0$ so that the bound above is any number no greater than 2 when $d$ is replaced by $\max\{d: d \in {\cal D}^+\}$.
\end{proof}

\begin{lem}
Under the assumptions of Theorem \ref{thm3.2'} there exists $\alpha_1 >0$ depending only on $\A$ such that for all $n$, all $1 \le k\le 2n/3$,
and all $0 \le \alpha \le \alpha_1$,
\beas
\E \exp\left(\alpha S_k^2/k\right)\le \exp\left(1 + \frac{3\alpha S_n^2}{4n}\right).
\enas
\label{lem3.5'}
\end{lem}
\begin{proof}
The steps are the same as in the proof of Lemma 3.5 of \cite{chatterjee2012new}.
For $Z$ a standard normal random variable independent of $\pi$, by  definition \eqref{eq:defWk} of $W_k$ we have
\begin{equation*}
\begin{split}
\E \exp(\alpha S_k^2/k) & = \E \exp\left(\sqrt \frac{2\alpha}{k}S_k Z \right) \\
& = \E \exp\left(\sqrt \frac{2\alpha}{k} W_k Z + \sqrt \frac{2\alpha}{k} \frac{kS_n}{n} Z\right).
\end{split}
\end{equation*}
By \eqref{eq9}, with $B$ given by \eqref{def:B.is.max}, for the first term we obtain the bound
\beas
\E \left[\exp \left(\sqrt{2\alpha} Z W_k/\sqrt{k}\right) \Bvert Z\right] \le \exp(2 \alpha B^2 Z^2). 
\enas
Thus,
\beas
\E \exp(\alpha S_k^2/k) \le \E \exp\left(2\alpha B^2 Z^2 + \sqrt \frac{2\alpha}{k} \frac{kS_n}{n} Z \right).
\enas
Recalling $S_n$ is nonrandom, using the non central $\chi_1^2$ identity \eqref{chi}, we  find 
that
\beas
\E \exp(\alpha S_k^2/k) \le \frac{1}{\sqrt {1-4 \alpha B^2}} \exp\left(\frac{\alpha k S_n^2}{(1-4\alpha B^2 ) n^2}\right)
\qmq{for $0 < \alpha < 1/(4B^2)$.}
\enas
The proof of the lemma is now completed by bounding $k$ by $2n/3$ and choosing $\alpha_1>0$ small enough so that $1/(1-4\alpha_1 B^2) $ is sufficiently close to 1.
\end{proof}
\noindent \textit{Proof of Theorem \ref{thm3.2'}:} We assume $\theta >0$. Applying our convention that zero variance normal random variables are equal to their mean almost surely, when $n=1$ we have $S_0=W_0=W_1=Z_0=Z_1=0$ and  the result holds trivially, so we assume $n \ge 2$. Recalling definition \eqref{eq:defcalD} of ${\cal D}^+$
for each $d>0$ in ${\cal D}^+$ and that ${\cal D}$ is symmetric, let $Y_d$ have the uniform $\U[-d/2,d/2]$ distribution, and be independent of each other and of the uniform random permutation $\pi$, and for $d=0$ let $Y_0=0$. Set
\beas
Y = \sum\limits_{d \in {\cal D}^+} Y_d.
\enas

For arbitrary $i,j$ satisfying $1\le i \le k <j \le n$ let ${\cal F}_{ij}=\sigma\{\pi(l): l \not \in \{i,j\}\}$.
Regarding the collection $\{\e_1,\ldots,\e_n\}$ as a multiset, we have
\beas 
\{\varepsilon_{\pi(i)},\varepsilon_{\pi(j)}\} = \{\varepsilon_i, i=1,\ldots, n\} \setminus  \{\varepsilon_{\pi(l)}, l \not \in \{i,j\}\},
\enas
showing that $\{\varepsilon_{\pi(i)},\varepsilon_{\pi(j)}\}$, and therefore also $\varepsilon_{\pi(i)}+\varepsilon_{\pi(j)}$
and $d_{ij}:=|\varepsilon_{\pi(i)}-\varepsilon_{\pi(j)}|$
are measurable with respect to ${\cal F}_{ij}$.
Further, the conditional distribution of
\beas
X_{ij}:=\frac{\varepsilon_{\pi(i)} - \varepsilon_{\pi(j)}}{2}
\enas
given ${\cal F}_{ij}$ is uniform over the set $\{-d_{ij}/2,d_{ij}/2\}$.

Let $S_k^{(i)}=S_k-\varepsilon_{\pi(i)}, W_k^{(i)}= S_k^{(i)}-(k/n)S_n$ and $Y^{(ij)}=Y-Y_{d_{ij}}$. 
For $\varepsilon_{\pi(i)}\not =\varepsilon_{\pi(j)}$, applying Lemma \ref{lem3.3'} and the easily verified fact that 
the zero bias distribution of the variable that takes the values $\{-a,a\}$ with equal probability
is uniform over $[-a,a]$, for some fixed Lipschitz function $f$,
we have
\begin{align*}
\E[&(\varepsilon_{\pi(i)}-\varepsilon_{\pi(j)}) f(W_k+Y)| {\cal F}_{ij}  ] \\
=&  2\E[X_{ij}f(X_{ij}+Y_{d_{ij}}+W_k^{(i)}+(\varepsilon_{\pi(i)}+\varepsilon_{\pi(j)})/2+ Y^{(ij)})|  {\cal F}_{ij} ]\\
=&  2\E[(X_{ij}^2-X_{ij}Y_{d_{ij}})f'(X_{ij}+Y_{d_{ij}}+W_k^{(i)}+(\varepsilon_{\pi(i)}+\varepsilon_{\pi(j)})/2+ Y^{(ij)})| {\cal F}_{ij} ]\\
=&  2\E[(X_{ij}^2-X_{ij}Y_{d_{ij}})f'(W_k+Y)|{\cal F}_{ij}] \\
=&  2\E[(d_{ij}^2/4-X_{ij}Y_{d_{ij}})f'(W_k+Y)| {\cal F}_{ij} ].
\end{align*}
We note that the equality between the first and final terms above holds also when $\varepsilon_{\pi(i)}=\varepsilon_{\pi(j)}$, both sides being zero. Taking expectation we obtain
\bea \label{eq:W_kpart}
\E[(\varepsilon_{\pi(i)}-\varepsilon_{\pi(j)}) f(W_k+Y)] = \E[t_{ij}f'(W_k+Y)]
\ena
where
\bea \label{def:tij}
t_{ij}= 2\left(\frac{d_{ij}^2}{4}-X_{ij}Y_{d_{ij}}\right) = \frac{\varepsilon_{\pi(i)}^2+\varepsilon_{\pi(j)}^2}{2}-\varepsilon_{\pi(i)}\varepsilon_{\pi(j)} -(\varepsilon_{\pi(i)}-\varepsilon_{\pi(j)})Y_{d_{ij}}.
\ena

It is easy to verify using \eqref{eq4}, or by integration by parts, that for $U \sim {\U}[-a,a]$,
\beas
\E[Uf(U)]=\frac{1}{2}\E[(a^2-U^2)f'(U)],
\enas
implying
\begin{multline} \label{eq:Ypart}
 \E [Y f(W_k+Y)]= \sum\limits_{d \in {\cal D}^+} \E[Y_d f(Y_d + W_k + (Y-Y_d)]\\
= \frac{1}{2}\sum\limits_{d \in {\cal D}^+} \E\left[\left(\frac{d^2}{4}-Y_d^2\right) f'(Y_d + W_k + (Y-Y_d)\right]
=\E[R_4 f'(W_k+Y)],
\end{multline}
where
\beas
R_4= \frac{1}{2}\sum\limits_{d \in {\cal D}^+}\left(\frac{d^2}{4}-Y_d^2\right).
\enas
Since ${\cal D}^+$ is finite there exists $C_0>0$ so that
\bea\label{eq12}
|R_4| \le C_0.
\ena
From \eqref{eq:WkDblSum},
\beas 
W_k=\frac{1}{n}\sum\limits_{i=1}^k \sum\limits_{j=k+1}^n (\varepsilon_{\pi(i)}-\varepsilon_{\pi(j)}),
\enas
so lettting
\bea \label{eq:defnWtilde.in.proof}
\widetilde{W}_k = W_k+Y
\ena
and combining \eqref{eq:W_kpart} and \eqref{eq:Ypart}, we have
\bea\label{eq14}
\E[\widetilde{W}_k f(\widetilde{W}_k)]=\E[T f'(\widetilde{W}_k)],
\ena
where the Stein coefficient $T$, in light of \eqref{def:tij},
is given by
\beas
T=\frac{1}{n}\sum\limits_{i=1}^k \sum\limits_{j=k+1}^n t_{ij}+R_4
= R_1- R_2-R_3 + R_4,
\enas
where
\beas
R_1=\frac{1}{2n} \left((n-k) \sum\limits_{i=1}^k \varepsilon_{\pi(i)}^2 + k \sum\limits_{j=k+1}^n \varepsilon_{\pi(j)}^2  \right), \quad
R_2=\frac{1}{n}\sum\limits_{i=1}^k \varepsilon_{\pi(i)} \sum\limits_{j=k+1}^n \varepsilon_{\pi(j)},
\enas
and
\begin{multline*}
R_3=\frac{1}{n}\sum\limits_{1 \le i \le k < j \le n}
 (\epsilon_{\pi(i)}-\epsilon_{\pi(j)})Y_{d_{ij}}\\
 =
 \sum_{d \in {\cal D}^+} Y_d \frac{1}{n}\sum_{1 \le i \le k < j \le n}  (\epsilon_{\pi(i)}-\epsilon_{\pi(j)})\mathds{1}(|\epsilon_{\pi(i)}-\epsilon_{\pi(j)}|=d)
=\sum\limits_{d \in {\cal D}^+} Y_d W_{k,d},
\end{multline*}
with $W_{k,d}$ as in \eqref{eq:defWkd}.
Since $|Y_d|\le d/2$, we have
\bea\label{eq16}
|R_3|\le \sum\limits_{d \in {\cal D}^+} \frac{d}{2} |W_{k,d}|.
\ena

Recalling that $\nu>0$ is a given fixed number, and that
\beas
\gamma^2=\frac{1}{n}\sum\limits_{i=1}^n \varepsilon_i^2, \qmq{set}
\widetilde \sigma^2 = \frac{k(n-k)\gamma^2}{n} \qmq{and} \sigma^2 = \frac{k(n-k)\eta^2}{n},
\enas
for a positive constant $\eta \ge \nu$, noting that since $n \ge 2$ both $\sigma^2$ and $\widetilde \sigma^2$ are positive. 
Then,
\begin{multline} \label{eq:breakTinto4}
\frac{\left(T-\sigma^2 \right)^2}{\sigma^2} = \frac{n}{k(n-k)\eta^2}\left( R_1-\sigma^2 -R_2-R_3+R_4\right)^2 \\
\le \frac{n}{k(n-k)\nu^2}\left( R_1-\sigma^2 -R_2-R_3+R_4\right)^2.
\end{multline}
To bound this quantity, consider first
\begin{multline*}
\begin{split}
&R_1-\widetilde \sigma^2\\
=&\frac{1}{2n} \left((n-k) \sum\limits_{i=1}^k \varepsilon_{\pi(i)}^2 + k \sum\limits_{j=k+1}^n \varepsilon_{\pi(j)}^2 \right)  - \widetilde \sigma^2\\
=&\frac{1}{2n} \left((n-k) \sum\limits_{i=1}^k \varepsilon_{\pi(i)}^2 + k \sum\limits_{j=k+1}^n \varepsilon_{\pi(j)}^2 \right) - \frac{k(n-k)}{n^2} \sum\limits_{i=1}^n \varepsilon_{\pi(i)}^2 \\
=&\frac{1}{2n} \left((n-k) \sum\limits_{i=1}^k \varepsilon_{\pi(i)}^2 + k \sum\limits_{j=k+1}^n \varepsilon_{\pi(j)}^2  - \frac{2k(n-k)}{n} \sum\limits_{i=1}^n \e_{\pi(i)}^2\right) \\
=&\frac{1}{2n} \left((n-k) \left( \sum\limits_{i=1}^k \varepsilon_{\pi(i)}^2  - \frac{k}{n} \sum\limits_{i=1}^n \varepsilon_{\pi(i)}^2 \right) + k\left(  \sum\limits_{j=k+1}^n \varepsilon_{\pi(j)}^2 - \frac{n-k}{n} \sum\limits_{j=1}^n \varepsilon_{\pi(j)}^2 \right) \right) \\
=& \frac{1}{2n} \left((n-k) \left( \frac{n-k}{n}\sum\limits_{i=1}^k \varepsilon_{\pi(i)}^2  - \frac{k}{n} \sum\limits_{i=k+1}^n \varepsilon_{\pi(i)}^2 \right) + k\left(  \frac{k}{n} \sum\limits_{j=k+1}^n \varepsilon_{\pi(j)}^2 - \frac{n-k}{n} \sum\limits_{j=1}^k \varepsilon_{\pi(j)}^2 \right) \right) \\
=&\frac{1}{2n^2} \left( \left( (n-k)^2 - k(n-k)\right)\sum\limits_{i=1}^k \varepsilon_{\pi(i)}^2 - \left((n-k)k- k^2\right)\sum\limits_{j=k+1}^{n} \varepsilon_{\pi(j)}^2\right) \\
=&\frac{n-2k}{2n^2}\left( (n-k) \sum_{i=1}^k \varepsilon_{\pi(i)}^2 - k\sum_{j=k+1}^n \varepsilon_{\pi(j)}^2 \right) .
\end{split}
\end{multline*}
Hence, for all $k = 1,2,\ldots,n$, with $B$ as in \eqref{def:B.is.max}, 
\begin{equation*}
\begin{split}
|R_1-\sigma^2| &\le \frac{|n-2k|}{2n^2} \left( (n-k) \sum_{i=1}^k \varepsilon_{\pi(i)}^2 + k\sum_{i=k+1}^n \varepsilon_{\pi(j)}^2 \right) + |\widetilde \sigma^2 - \sigma^2| \\
& \le \frac{|n-2k|}{2} \gamma^2 + |\widetilde \sigma^2 - \sigma^2| \le \frac{|n-2k|}{2} B^2 + |\widetilde \sigma^2 - \sigma^2|.
\end{split}
\end{equation*}

Choosing $k$ such that $|n-2k| \le 1$, we obtain
\bea \label{eq17}
|R_1 -\sigma^2| \le \frac{B^2}{2} + |\widetilde \sigma^2 - \sigma^2|.
\ena

Regarding $R_2$, for any $k \in \{1,\ldots,n\}$ we have
\bea \label{eq18}
|R_2| = \frac{1}{n}\Bvert\sum\limits_{i=1}^k  \varepsilon_{\pi(i)} \sum\limits_{j=k+1}^n \varepsilon_{\pi(j)} \Bvert \le B|S_k|.
\ena
Hence, for $k$ such that $|2k-n| \le 1$, from \eqref{eq:breakTinto4}, \eqref{eq17}, \eqref{eq18}, \eqref{eq16} and \eqref{eq12}, and that $|\widetilde \sigma^2 - \sigma^2|=|\frac{k(n-k)}{n}(\gamma^2-\eta^2)|$, we obtain
\begin{equation*}
\begin{split}
\frac{(T - \sigma^2)^2}{\sigma^2} &\le\frac{n}{k(n-k){\nu}^2}\left(B^2/2 + |\widetilde \sigma^2 - \sigma^2| + B|S_k| + \sum\limits_{d \in {\cal D}^+} \frac{d}{2} |W_{k,d}| + C_0\right)^2\\
&\le C\left(1 + n(\gamma^2-\eta^2)^2 +\frac{S_k^2}{k} +\sum\limits_{d \in {\cal D}^+} \frac{W_{k,d}^2}{k}\right)
\end{split}
\end{equation*}
for some constant $C$ depending uniquely on $\A$ and $\nu$.

We now verify that the hypotheses of Theorem \ref{thm1} hold for ${\widetilde W}_{{k}}$ of \eqref{eq:defnWtilde.in.proof} and $T$.
Clearly ${\widetilde W}_k$  satisfies 
$\E({\widetilde W}_{k}) = 0$ and $\E({\widetilde W}_{k}^2)< \infty$. By \eqref{eq14}, $T$ is a Stein coefficient for ${\widetilde W}_{k}$, and $T$ is easily verified to be bounded. Writing $Z$ for short for $\eta Z_k$ in the statement of Theorem \ref{thm3.2'}, we note that $Z$ is distributed ${\cal N}(0,\sigma^2)$, and 
by Theorem \ref{thm1} we can construct a version of ${\widetilde W}_{{k}}$ and $Z$ on the same probability space so that for all $\theta$,
\begin{multline*}
\E \exp(\theta|{\widetilde W}_{k} - Z|)
\le 2\E \exp(2\theta ^2 \sigma^{-2}(T-\sigma^2)^2) \\
\le
2 \E \exp \left( 2 C \theta^2 \left( 1 + n(\gamma^2-\eta^2)^2+ \frac{S_k^2}{k} +\sum\limits_{d \in {\cal D}^+} \frac{W_{k,d}^2}{k} \right) \right). 
\end{multline*}
With $D=\frac{1}{2}\sum_{d \in {\cal D}^+} d$ we have
$|W_{k} - {\widetilde W}_{k}| \le |Y| \le \sum\limits_{d \in {\cal D}^+} |Y_d| \le D$. Letting $q=|{\cal D}^+|+1$, we have
\begin{align*}
& \E \exp (\theta |W_{k} - Z|) \\
&\le 2 \exp(D|\theta| + 2C\theta^2+ 2C\theta^2 n(\gamma^2-\eta^2)^2)\E \exp \left( 2C\theta^2 \frac{S_k^2}{k} +2C\theta^2 \sum\limits_{d \in {\cal D}^+} \frac{W_{k,d}^2}{k}\right)\\
&\le \frac{2}{q} \exp(D|\theta| + 2C\theta^2+ 2C\theta^2 n(\gamma^2-\eta^2)^2)\left( \E \exp \left( 2Cq\theta^2 \frac{S_k^2}{k}\right) +
\sum_{d \in {\cal D}^+} \E \exp \left(
2Cq\theta^2 \frac{W_{k,d}^2}{k}\right) \right),
\end{align*}
by the convexity of the exponential function. Using Lemmas \ref{lem3.5'} and \ref{lem3.4'}, there exists
$\theta_3>0$ depending only on $\A$ and $\nu$ such that for all $\theta \le \theta_3$, we obtain
\begin{equation*}
\begin{split}
\E \exp (\theta |W_{k} - Z|) &\le \frac{2}{q} \exp(D |\theta| + 2C\theta^2+ 2C\theta^2 n(\gamma^2-\eta^2)^2)\left( \exp \left(1+ 6Cq\theta^2 \frac{S_n^2}{4n}\right) +
2(q-1)\right)\\
& \le 2\exp(D |\theta| + 2C\theta^2+ 2C\theta^2 n(\gamma^2-\eta^2)^2)\left( \exp \left(1+ 6Cq\theta^2 \frac{S_n^2}{4n}\right) +
2\right).
\end{split}
\end{equation*}
Now choose $\theta_4 > 0$, depending only on ${\cal A}$ and $\nu$, so that 
\beas
2\exp(D \theta_4 + 2C \theta_4^2)\le e \qmq{and note} 
 \exp \left(1+ \theta^2 x\right) +
2 \le \exp \left(2+ \theta^2 x\right) \qm{for all $x \ge 0$,}
\enas 
implying
that for $\theta\le \theta_2:=\theta_3 \wedge \theta_4$,
\beas
\begin{split}
\E \exp (\theta |W_{k} - Z|) &\le& \exp \left(1+2C\theta^2 n(\gamma^2-\eta^2)^2 +2+6Cq\theta^2 \frac{S_n^2}{4n} \right)\\
&=& \exp \left(3 + 6Cq\theta^2 \frac{S_n^2}{4n} + 2C\theta^2 n(\gamma^2-\eta^2)^2\right),
\end{split}
\enas
which is the desired bound. \qed

\section{The Induction Step}\label{sec3}
In this section we present Theorem \ref{thm4.1'}, which we use to prove Theorem \ref{exchangeable} that generalizes Theorem 1.4 in \cite{chatterjee2012new}. Let $\epsilon_1, \epsilon_2, \dots \epsilon_n$ be arbitrary elements of a finite set $\A \subset \mathbb{R}$, not necessarily distinct, and let $\pi$ be a uniform random permutation of $\{1,2,\dots ,n \}$. For each $1\le k \le n$ recall
\bea \label{eq:SkWkind}
S_k=\sum_{i=1}^k \epsilon_{\pi(i)} \qmq{and} W_k=S_k - \frac{kS_n}{n}.
\ena

We show $(W_1,\ldots,W_n)$ and a positive multiple of a Gaussian vector $(Z_1,\ldots,Z_n)$ obtained by evaluating a Brownian bridge process on $[0,n]$ at integer time points can be coupled on the same space so that the moment generating function of their maximum absolute difference achieves the exponential bound \eqref{eq20} below. In place of coupling, the result of the theorem can be equivalently stated in terms of the existence of a joint probability function $\rho_{\pmb{\e}}^n(\s, \mz)$ on $(S_1,\ldots,S_n)$ and $(Z_1,\ldots,Z_n)$ having the correct marginals whose joint realization obeys the desired bound.

It will be helpful to regard the collection $\pmb{\e}=\{\e_1,\ldots, \e_n\}$ as a multiset.
We say $\s \in \mathbb{R}^n$ is a `path' corresponding to a multiset of `increments' $\pmb{\e}$ when there exists $\pi \in {\cal P}_n$, the set of permutations on $\{1,\ldots,n\}$, such that $\s$ can be achieved by summing the increments $\e$ in the order given by $\pi$, that is, when $\s$ is an element of the set of all feasible paths
\bea\label{eq22}
\A_{\pmb{\e}}^n := \{\s \in \mathbb{R}^n : s_k= \sum_{i=1}^k \e_{\pi(i)}, k=1,\ldots,n, \pi \in {\cal P}_n\}.
\ena
Conversely, the multiset of increments corresponding to a path $\s$ is given by
\bea \label{eq:def.incrementsfrompath}
\pmb{\e}^{\s}{=\{s_1,s_2-s_1,\ldots,s_n-s_{n-1}\}},
\ena
so that $\s \in {\cal A}_{\pmb{\e}}^n$ if and only if $\pmb{\e}^{\s}=\pmb{\e}$.

Suppose that among $\pmb{\e}$ are $l$ distinct numbers, appearing with multiplicities $m_1, \dots, m_l$, necessarily summing to $n$. Then letting $f_{\pmb{\e}}^n(\s)$ be the probability mass function of $(S_1,\ldots,S_n)$ as given by \eqref{eq:SkWkind}, we have
\bea\label{eq23}
|\A_{\pmb{\e}}^n| = \frac{n!}{m_1! m_2! \dots m_l!} \qmq{and} 
f_{\pmb{\e}}^n(\s) = \frac{1}{|\A_{\pmb{\e}}^n|} 
{\mathds{1}}(\s \in \A_{\pmb{\e}}^n)=\frac{1}{|\A_{\pmb{\e}}^n|} 
{\mathds{1}}(\pmb{\e}^{\s}=\pmb{\e}),
\ena
that is, the distribution $f_{\pmb{\e}}^n(\s)$ is uniform
over $\A_{\pmb{\e}}^n$.

The following result is a conditional version of Theorem \ref{exchangeable}.
\begin{thm}\label{thm4.1'}
Let $\epsilon_1, \epsilon_2, \dots \epsilon_n$ be arbitrary elements of a finite set $\A \subset \mathbb{R}$, not necessarily distinct, $\pi$ a uniform random permutation of $\{1,2,\dots,n\}$,
$S_k$ and $W_k$ as in \eqref{eq:SkWkind}, and 
$\gamma^2 =n^{-1} \sum_{i=1}^n \e_i^2$.
Then there exists a positive universal constant $C$, and for every $\nu>0$ positive constants $K_1, K_2$ and $\lambda_0$ depending only on $\A$ and $\nu$ such that for any integer $n \ge 1$ and every $\eta \ge \nu$ one may construct a version of ${(W_k)}_{0\le k \le n}$ and Gaussian random variables ${(Z_k)}_{0\le k \le n}$ with zero mean and covariance
\bea\label{eq19}
{\rm Cov}(Z_i, Z_j) = \frac{(i\wedge j)(n - (i\vee j))}{n}
\ena
on the same probability space such that
\begin{multline} \label{eq20}
\E \exp(\lambda \max_{0 \le i \le n} |W_i - \eta Z_i|) 
\\ \le \exp\left(C \log n + \frac{K_1 \lambda^2 S_n ^2}{n} +K_2\lambda^2 n(\gamma^2-\eta^2)^2\right) \qmq{for any $\lambda \le \lambda_0$.}
\end{multline}
\end{thm}
\begin{proof}
As the result holds trivially for $\lambda \le 0$ we need consider only $\lambda > 0$. Also, as $W_0=0$ and $Z_0=0$ by convention it suffices to consider the maximum over $1 \le i \le n$ in \eqref{eq20}.
We use Theorem \ref{thm3.2'} and induction to prove the theorem.

Recall the constants $\alpha_1$ from Lemma $\ref{lem3.5'}$ depending only on ${\cal A}$, and $c_1, c_2$ and $\theta_2$ from Theorem $\ref{thm3.2'}$, depending only on ${\cal A}$ and $\nu$.
With $B$ given in \eqref{def:B.is.max},
letting $\theta_5$ be the unique positive solution to
\bea\label{theta1}
\frac{1}{\sqrt{1- B^4 \theta^2/2}}=\frac{4}{3},
\ena
depending only on ${\cal A}$. We will demonstrate the claim holds with
\bea\label{eq21}
C = \frac{2 + \log 4}{\log (3/2)}, \quad 
K_1=8c_1, \quad K_2=18c_2 \qmq{and} \lambda_0 = \sqrt {\frac{\alpha_1}{32c_1}} \wedge \frac{\theta_2}{2} \wedge \frac{\theta_5}{\sqrt {72c_2}}.
\ena

Note that any multiset $\pmb{\e}=\{\e_1,\ldots,\e_n\}$ of elements of ${\cal A}$ lies in exactly one set of the form
\beas
{\cal B}^n(a,b) = \{ \{\e_1, \e_2, \dots, \e_n\}: \sum_{i=1}^n \e_i = a, \frac{1}{n}\sum_{i=1}^n \e_i^2=b^2 \}
\enas
as $a$ and $b$ range over all pairs of feasible values of $S_n$ and $\gamma$, respectively. Fix one such feasible pair $a,b$, which may be notationally suppressed when clear from context, let $\pmb{\e} \in {\cal B}^n(a,b)$ be arbitrary and fix any value $\eta>0$.

With $f_{\pmb{\e}}^n(\s)$ the probability mass function of $(S_1,\ldots,S_n)$ given in \eqref{eq23}
and $\phi^n (\mz)$ the probability density function of a Gaussian random vector $(Z_1, \dots, Z_n)$ with mean zero and covariance \eqref{eq19}, we show that for each $n \ge 1$, we can construct a joint probability function $\rho_{\pmb{\e}}^n(\s, \mz)$ on $\A_{\pmb{\e}}^n \times \mathbb{R}^n$ having the desired marginals
\bea\label{eq24}
\sum_{\s \in \A_{\pmb{\e}}^n} \rho_{\pmb{\e}}^n(\s, \mz) = \phi^n (\mz) \qmq{and} \int_{\mathbb{R}^n} \rho_{\pmb{\e}}^n(\s, \mz) d\mz = f_{\pmb{\e}}^n(\s)
\ena
and satisfying the exponential bound
\begin{multline} \label{rho.mgf.bound}
\lefteqn{\int_{{\mathbb{R}^n}} \sum_{\s \in \A_{\pmb{\e}}^n} \left[
\exp\left(\lambda \max_{1 \le i \le n} \Bvert s_i - \frac{ia}{n} -\eta z_i \Bvert \right) \rho_{\pmb{\e}}^n(\s, \mz)\right] d\mz} \\
\le \exp\left(C \log n + \frac{K_1 \lambda^2 a^2}{n} +K_2\lambda^2 n(b^2-\eta^2)^2\right) 
\qmq{for all $\lambda \in (0,\lambda_0]$,}
\end{multline}
for all $\eta \ge \nu$, with $C,K_1,K_2$ and $\lambda_0$ as in \eqref{eq21}, with $C$ universal and the latter three constants depending only on ${\cal A}$ and $\nu$.

We will prove the claim by induction on $n$. For $n=1$ we note that $W_1=0$ by \eqref{eq:defWk} and $Z_1=0$ by convention, since it has mean zero and covariance given by \eqref{eq19}. Hence \eqref{eq20} holds for $n=1$ for all $C$, all nonnegative $K_1,K_2$, and all $\lambda_{0}$, and in particular for the set of constants specified in \eqref{eq21}.

Given $n \ge 2$, suppose that for all $l= 1,2,\dots, n-1$ and all multisubsets $\pmb{\zeta}$ of ${\cal A}$ of size $l$ we can construct $\rho_{\pmb{\zeta}}^l(\s, \mz)$ satisfying \eqref{eq24} and \eqref{rho.mgf.bound}.
Take $k=[n/2]$, 
let $\sqcup$ denote multiset union and define the sets
\begin{multline*}
{{\cal S}_{\pmb{\e}}^{n,k} = \{s:  \sum_{\e \in \pmb{\e}_1}\e=s \text{  for some  } \pmb{\e}_1, \pmb{\e}_2 \text{ such that } |\pmb{\e}_1|=k,  \pmb{\e}_1 \sqcup \pmb{\e}_2= \pmb{\e}\},} \qm{and}\\
{\cal B}_{{\pmb{\e}}}^{n,{k}}(s)  = \{(\pmb{\e}_1, \pmb{\e}_2): \sum_{\e \in \pmb{\e}_1}\e=s,  |\pmb{\e_1}|=k, \pmb{\e}_1 \sqcup \pmb{\e}_2= \pmb{\e}\} \qm{for $s \in {\cal S}_{\pmb{\e}}^{n,k}$}.
\end{multline*}
That is, ${\cal S}_{\pmb{\e}}^{n,k}$ is the set of all feasible values at time $k$ of a path having increments $\pmb{\e}$, and ${\cal B}_{{\pmb{\e}}}^{n,{k}}(s)$ is the set of all ways of dividing the $n$ increments $\pmb{\e}$ into sets of sizes $k$ and $n-k$ so that the path at time $k$ takes the value $s$.
Counting the number of paths that take the value
$s    \in {\cal S}_{\pmb{\e}}^{n,k}   $
at time $k$ shows that $g_{\pmb{\e}}^{n,k}(s)$, the marginal density of $S_k$ in $f_{\pmb{\e}}^n(\s)$, is given by
\bea\label{eq25}
g_{\pmb{\e}}^{n,k}(s)= \frac{ \sum_{(\pmb{\zeta}_1, \pmb{\zeta}_2) \in {\cal B}_{{\pmb{\e}}}^{n,{k}}(s)} |\A_{\pmb{\zeta}_1}^k| |\A_{\pmb{\zeta}_2}^{n-k}|}{|\A_{\pmb{\e}}^n|}.
\ena 

Similarly, let $h^{n,k}(z)$ denote the marginal density function of $Z_k$ in $\phi^n (\mz)$, that of the Gaussian distribution with mean zero and variance $k(n-k)/n$. By Theorem \ref{thm3.2'},
there exists a joint density function $\psi_{\pmb{\e}}^{n,k}(s,z)$ on ${{\cal S}_{\pmb{\e}}^{n,k}}\times \mathbb{R}$ and positive constants $c_1, c_2$ and $\theta_2$, depending only on ${\cal A}$ and $\nu$, such that
 \bea\label{eq26}
 \int \psi_{\pmb{\e}}^{n,k}(s,z) dz = g_{\pmb{\e}}^{n,k}(s), \quad \sum_{s \in {{\cal S}_{\pmb{\e}}^{n,k}} } \psi_{\pmb{\e}}^{n,k}(s,z) = h^{n,k}(z),
 \ena
 and for $\theta \le \theta_2$ and $\eta \ge \nu$,
 \bea\label{eq27}
 \int \sum_{s \in {{\cal S}_{\pmb{\e}}^{n,k}} } \left[\exp\left(\theta \Bvert s-\frac{ka}{n} -\eta z \Bvert \right)\psi_{\pmb{\e}}^{n,k}(s,z)\right]  dz \le \exp\left(3 + \frac{c_1\theta^2 a^2}{n} +c_2\theta^2n(b^2-\eta^2)^2\right).
 \ena

For $s {\in {\cal S}_{\pmb{\e}}^{n,k}}, z \in \mathbb{R}$, and recalling the definition \eqref{eq:def.incrementsfrompath} of $\pmb{\e}^{\s}$,
$\s^1,\s^2$ such that $
(\pmb{\e}^{\s^1},\pmb{\e}^{\s^2}) 
\in B_{\pmb{\e}}^{n,k}(s)$, 
$\mz^1 \in \mathbb{R}^k$ and $\mz^2 \in \mathbb{R}^{n-k}$,
let
 \begin{equation}
\label{eq28}
\gamma_{\pmb{\e}}^n(s,z,\s^1,\mz^1, \s^2, \mz^2) = \psi_{\pmb{\e}}^{n,k}(s,z) P_{\pmb{\e},s}(\pmb{\e}^{\s^1}, \pmb{\e}^{\s^2})\rho_{\pmb{\e}^{\s^1}}^k(\s^1, \mz^1)\rho_{\pmb{\e}^{\s^2}}^{n-k}(\s^2, \mz^2)
\end{equation}
where
\beas
P_{\pmb{\e},s}(\pmb{\e}_1, \pmb{\e}_2)=\frac{|\A_{\pmb{\e}_1}^k| |\A_{\pmb{\e}_2}^{n-k}|}
 {\sum_{(\pmb{\zeta}_1, \pmb{\zeta}_2) \in {B^{n,k}_{\pmb{\e}}(s)}}|\A_{{\pmb{\zeta}}_1}^k| |\A_{{\pmb{\zeta}}_2}^{n-k}|}{\mathds{1}}((\pmb{\e}_{1}, \pmb{\e}_{2}) \in { {\cal B}_{\pmb{\e}}^{n,k}(s)}).
 \enas

Interpreting \eqref{eq28} in terms of a construction, one first samples the joint values $s$ and $z$ of the coupled random walk and Gaussian path at time $k$, then chooses increments corresponding to $\s^1$ and $\s^2$, the first and last half of the walk according to their likelihood over the choices of those whose increments over the first half of the walk sum to $s$, and whose union of increments over both halves must be $\pmb{\e}$, and then samples coupled values of the paths with discrete Brownian bridges before and after time $k$.

One may verify that $\gamma_{\pmb{\e}}^n$ is a density function by integrating over $\mz^1$ and $\mz^2$ using the second equality in \eqref{eq24} followed by applying the second equality in \eqref{eq23}, integrating over $z$, and then
summing over all $\s^1$ and $\s^2$ and $s$, this last operation being equivalent to summing over all paths $\s$ with increments $\pmb{\e}$, see  \eqref{eq:also.verifies.density} and the explanation following.

Now, let $(S, Z, \mathbf{S}^1, \mathbf{Z}^1, \mathbf{S}^2, \mathbf{Z}^2)$ be a random vector with density $\gamma_{\pmb{\e}}^n$ where $\mathbf{S}^1 = {(S_i^1)}_{1 \le i \le k}$, $\mathbf{S}^2 = {(S_i^2)}_{1 \le i \le n- k}$ and $\mathbf{Z}^1 = {(Z_i^1)}_{1 \le i \le k}$, $\mathbf{Z}^2 = {(Z_i^2)}_{1 \le i \le n-k}$.  Let  $\mathbf{S}$ be obtained by `piecing' the paths $\mathbf{S}^1$ and $\mathbf{S}^2$ together at time $k$ according to the rule
\bea\label{U_i}
S_i = \left\{
\begin{array}{cc}
	S_{i}^1 & 1 \le i \le k \\
	S+S_{i-k}^2 & k<i \le n,
\end{array}
\right.
\ena
here noting $S_k= S$,
and define $\mathbf{Z}$ by
\bea \label{eq:defYi}
Z_i = \left\{
\begin{array}{cc}
{Z}_i^1 + \frac{i}{k} Z & 1 \le i \le k \\
 {Z}_{i-k}^2 + \frac{n-i}{n-k} Z & k< i \le n,
\end{array}
\right.
\ena
here noting likewise that $Z_k=Z$, since $Z_k^1= 0$. Now as in \cite{chatterjee2012new}, we demonstrate that 
$\rho_{\pmb{\e}}^n({\textbf s},{\textbf z})$, the joint density of $(\mathbf{S}, \mathbf{Z})$, achieves the desired marginals \eqref{eq24} and exponential bound
\eqref{rho.mgf.bound}.

\noindent 1. \textit{Marginal distribution of} $\mathbf{S}$.
Let $\pmb{s}$ be the path constructed from $s, \s^1$ and $\s^2$ as $\mathbf{S}$ is constructed from $S,\mathbf{S}^1$ and $\mathbf{S}^2$ in \eqref{U_i}. 
Note that 
\beas
\{\s:\s \in {\cal A}_{\pmb{\e}}^n\} = \{\s: (\pmb{\e^{\s^1}},\pmb{\e^{\s^{2}}}) \in {\cal B}_{\pmb{\e}}^{n,k}(s_k)\},
\enas
and that $S_k=S$ almost surely. Hence, if $\mathbf{S} \not \in {\cal A}_{\pmb{\e}}^n$ then from \eqref{eq28} $\mathbf{S}$ has probability zero.
For the marginal of $\gamma_{\pmb{\e}}^n$ to be non-zero on $s,\s^1,\s^2$, first $s$ must be a feasible value at time $k$ for a path with increments $\pmb{\e}$, then $\s^1$ must be a path of increments that attains the value $s$ at time $k$, and finally the collection of increments determined by $\s^1$ and $\s^2$ must match the given set $\pmb{\e}$ of increments. In this case we obtain from \eqref{eq24}, \eqref{eq26} and \eqref{eq25}, that the marginal distribution of $(S,\mathbf{S}^1,\mathbf{S}^2)$ is given by
\begin{equation}\label{eq:also.verifies.density}
\begin{split}
&  \int \gamma_{\pmb{\e}}^n(s,z,\s^1,\mz^1, \s^2, \mz^2) d\mz^2  d\mz^1  dz \\
&= g_{\pmb{\e}}^{n,k}(s) P_{\pmb{\e},s}(\pmb{\e}^{\s^1}, \pmb{\e}^{\s^2}) f_{{\pmb{\e}}^{\s^1}}^k (\s^1) f_{{\pmb{\e}}^{\s^2}}^{n-k}(\s^2)\\
&= \frac{ \sum_{(\pmb{\zeta}_1, \pmb{\zeta}_2) \in {{\cal B}^{n,k}_{\pmb{\e}}(s)}  } |\A_{\pmb{\zeta}_1}^k| |\A_{\pmb{\zeta}_2}^{n-k}|}{ |\A_{\pmb{\e}}^n|}\frac{|\A_{\pmb{\e}^{\s^1}}^k| |\A_{\pmb{\e}^{\s^2}}^{n-k}|}
 {\sum_{(\pmb{\zeta}_1, \pmb{\zeta}_2) \in {{\cal B}^{n,k}_{\pmb{\e}}(s)}   } |\A_{\pmb{\zeta}_1}^k| |\A_{\pmb{\zeta}_2}^{n-k}|} \frac{1}{|\A_{{\pmb{\e}}^{\s^1}}^k||\A_{{\pmb{\e}}^{\s^2}}^{n-k}|} \\
& = \frac{1}{|\A_{\pmb{\e}}^n|}\\
&= f_{\pmb{\e}}^n(\pmb{s}).
\end{split}
\end{equation}
Now observing that \eqref{U_i} gives a one-to-one correspondence between $(S, \mathbf{S}^1, \mathbf{S}^2)$ and $\mathbf{S}$
we find that $\mathbf{S}$ has marginal density $f_{\pmb{\e}}^n(\pmb{s})$ as in \eqref{eq23}.\\
\\
2. \textit{Marginal distribution of} $\mathbf{Z}$. Consider
$
\A_{\pmb{\e_1}}^k \times \A_{\pmb{\e_2}}^{n-k},
$ 
the set of all pairs of paths $(\s^1, \s^2)$ with increments $\pmb{\e}_1$ and $\pmb{\e}_2$ respectively. 
 Using \eqref{eq24} and \eqref{eq26}, and noting that $(\pmb{\e}^{\s^1}, \pmb{\e}^{\s^2})=(\pmb{\e}_1,\pmb{\e}_2)$ for $(\s^1,\s^2) \in \A_{\pmb{\e_1}}^k \times \A_{\pmb{\e_2}}^{n-k}$, the marginal distribution of $Z$, $\mathbf{Z}^1$, $\mathbf{Z}^2$ is given by
\beas 
\lefteqn{ \sum_{s   {\in {\cal S}_{\pmb{\e}}^{n,k}}  }~ \sum_{(\pmb{\e}_1,\pmb{\e}_2)  {\in {\cal B}_{\pmb{\e}}^{n,k}(s)}   }~ \sum_{(\s^1,\s^2)\in {\A_{\pmb{\e_1}}^k \times \A_{\pmb{\e_2}}^{n-k}}} \gamma_{\pmb{\e}}^n(s,z,\s^1,\mz^2, \s^2, \mz^2)} \nonumber \\
& =& \sum_{s  {\in {\cal S}_{\pmb{\e}}^{n,k}}  } \psi_{\pmb{\e}}^{n,k}(s,z)  \left[ \sum\limits_{(\pmb{\e}_1,\pmb{\e}_2) \in  {{\cal B}_{\pmb{\e}}^{n,k}(s)}    } P_{\pmb{\e},s}(\pmb{\e}_1,\pmb{\e}_2) \sum_{(\s^1,\s^2)\in {\A_{\pmb{\e_1}}^k \times \A_{\pmb{\e_2}}^{n-k}}}
\rho_{\pmb{\e}_1}^k(\s^1, \mz^1)\rho_{\pmb{\e}_2}^{n-k}(\s^2, \mz^2)  \right] \nonumber\\
& =&\sum_{s {\in {\cal S}_{\pmb{\e}}^{n,k}} } \psi_{\pmb{\e}}^{n,k}(s,z)  \left[ \sum\limits_{(\pmb{\e}_1,\pmb{\e}_2) \in  {{\cal B}_{\pmb{\e}}^{n,k}(s)}      } P_{\pmb{\e},s}(\pmb{\e}_1,\pmb{\e}_2) \sum_{\s^1 \in \A_{ \pmb{\e}_1}^k}
\rho_{\pmb{\e}_1}^k(\s^1, \mz^1)  \sum_{\s^2 \in \A_{\pmb{\e}_2}^{n-k}} \rho_{\pmb{\e}_2}^{n-k}(\s^2, \mz^2)  \right] \nonumber \\
& =&\sum_{s {\in {\cal S}_{\pmb{\e}}^{n,k}} } \psi_{\pmb{\e}}^{n,k}(s,z)  \left[
\sum\limits_{(\pmb{\e}_1,\pmb{\e}_2) \in {{\cal B}_{\pmb{\e}}^{n,k}(s)}     } P_{\pmb{\e},s}(\pmb{\e}_1,\pmb{\e}_2)
\phi^k(\mz^1)\phi^{n-k}(\mz^2)   \right] \nonumber \\
&=& \phi^k(\mz^1)\phi^{n-k}(\mz^2) \sum_{s} \psi_{\pmb{\e}}^{n,k}(s,z) \nonumber\\
& =& \phi^{n-k}(\mz^2)  \phi^k(\mz^1) h^{n,k}(z)
\enas
where we have used that $\sum_{(\pmb{\e}_1,\pmb{\e}_2) \in {\cal B}^{n,k}_{\pmb{\e}}(s)} 
P_{\pmb{\e},s}(\pmb{\e}_1,\pmb{\e}_2)=1.$
Hence $Z$, $\mathbf{Z}^1$ and $\mathbf{Z}^2$ are independent with densities $h^{n,k}(z)$, $\phi^k({{\textbf z}^1})$ and $\phi^{n-k}({\textbf z}^2)$ respectively, implying that $\mathbf{Z}$ given by \eqref{eq:defYi} is a multivariate mean zero Gaussian random vector. As in \cite{chatterjee2012new}, one can verify that $\mathbf{Z}$ has covariances given by \eqref{eq19}, and hence $\mathbf{Z} \sim \phi^n({\textbf z}).$ \\ 
\\ 3.\textit{The exponential bound.} For $1\le i\le n$, letting
\beas
W_i = S_i - \frac{ia}{n},
\enas
we show that 
\beas
\E \exp(\lambda \max_{1 \le i\le n} |W_i - \eta Z_i|) \le \exp \left( C\log n +\frac{K_1 \lambda^2 a}{n} +K_2\lambda^2 n(b^2-\eta^2)^2\right) \qmq{for $\lambda \in (0,\lambda_0]$}
\enas
where $C, K_1,K_2$ and $\lambda_0$ are as in \eqref{eq21}. 
We continue to  proceed as in \cite{chatterjee2012new}.

Again writing $S$ for $S_k$, let
\beas
T_L := \max_{1 \le i\le k}\Bvert S_i^1 - \frac{iS}{k} - \eta Z_i^1 \Bvert, T_R := \max_{k < i \le n}\Bvert S_{i-k}^2 - \frac{i-k}{n-k}(a-S) - \eta Z_{i-k}^2 \Bvert, 
\enas
and
\beas
T := \Bvert S - \frac{ka}{n} -\eta Z \Bvert.
\enas
Note that when $1 \le i\le k$ we have
\begin{equation*}
\begin{split}
|W_i - \eta Z_i| & = \Bvert S_i^1- \frac{ia}{n} - \eta \left(Z_i^1 + \frac{iZ}{k} \right) \Bvert \\
& \le  \Bvert S_i^1- \frac{iS}{k} -\eta Z_i^1 \Bvert +  \Bvert \frac{iS}{k} - \frac{ia}{n} -  \frac{i}{k}\eta Z\Bvert\\
& \le T_L + \frac{i}{k} T\le T_L+T.
\end{split}
\end{equation*}
Similarly for $k <i \le n$ one can verify
$|W_i - \eta Z_i| \le T_R+T$, proving
\beas 
\max_{1 \le i\le n}|W_i - \eta Z_i| \le \max\{T_L+T, T_R+T\}.
\enas
Now fixing $\lambda \le \lambda_0$, the inequality $\exp(x \vee y) \le e^x + e^y$ yields
\bea\label{eq30}
\exp(\lambda \max_{1 \le i\le n} |W_i - \eta Z_i|) \le \exp(\lambda T_L +\lambda T ) + \exp(\lambda T_R + \lambda T).
\ena
To prove that the exponential bound holds, we develop inequalities on the expectation of the two quantities on the right hand side of \eqref{eq30}, starting with the expression involving $T_L$.

Note that $\pmb{\e}^{\s^1}$ determines $S$, and since $\pmb{\e}$ is fixed $\pmb{\e}^{\s^2}$ is also determined, so by \eqref{eq28} the conditional density of $(\mathbf{S}^1, \mathbf{Z}^1)$ given $(\pmb{\e}^{\mathbf{S}^1}, Z)$ is $\rho_{\pmb{\e}^{\mathbf{S}^1}}^k(\s^1,{{\textbf z}^1})$. Now using that the moment generating functions of $T_L$ and $T$ are finite everywhere and that $T$ is a function of $\{S,Z\}$, invoking the induction hypothesis and applying the Cauchy-Schwarz inequality twice, with $\gamma_1^2=(1/k)\sum_{i=1}^k \e_{\pi(i)}^2$ we obtain
\begin{multline} \label{eq:Cauchy.twice}
\E \exp(\lambda T_L +\lambda T) = \E \left[\E \left( \exp(\lambda T_L)|\pmb{\e}^{\mathbf{S}^1},Z\right) \exp(\lambda T) \right] \\ 
 \le \left[ \E \left(\E \left( \exp(\lambda T_L)|\pmb{\e}^{\mathbf{S}^1},Z\right)^2 \right) \E(\exp(2\lambda T)) \right]^{1/2}\\
 \le \exp(C\log k) \left[ \E \exp \left( \frac{2K_1\lambda^2 S^2}{k} +2K_2 \lambda^2 k(\gamma_1^2-\eta^2)^2 \right)  \E \exp(2\lambda T) \right]^{1/2}\\
  \le \exp(C\log k) \left[\E \exp\left(\frac{4K_1 \lambda^2 S^2}{k}\right)\E \exp\left(4K_2\lambda^2 k(\gamma_1^2-\eta^2)^2 \right)\right]^{1/4} \left(\E \exp(2\lambda T)\right)^{1/2}.
\end{multline}

For the first expectation in \eqref{eq:Cauchy.twice},  \eqref{eq21} implies that $0 \le 4K_1 \lambda^2 \le \alpha_1$, and as $|2k-n|\le 1$ we may invoke Lemma $\ref{lem3.5'}$ to yield
\bea \label{eq:first.expectation}
\E \exp\left( \frac{4K_1 \lambda^2 S^2}{k}\right) \le \exp\left( 1 + \frac{3 K_1\lambda^2 a^2}{n}\right).
\ena
For the second expectation in \eqref{eq:Cauchy.twice}, recalling the definition of $\gamma_1^2$,
\bea\label{etabd}
\E \exp\left(4K_2 \lambda^2 k(\gamma_1^2-\eta^2)^2\right)=\E \exp \left(4K_2\lambda^2 \frac{1}{k}{\left(\sum\limits_{i=1}^k (\e_{\pi(i)}^2-\eta^2)\right)^2}\right) = \E \exp \left(\theta^2 \frac{U_k^2}{k}\right), 
\ena
where  $\theta=2 \lambda \sqrt{K_2}$, and we write 
\beas
U_k = \sum_{i=1}^k (\e_{\pi(i)}^2-\eta^2)  = \sum_{i=1}^n \left( \e_i^2 \mathds{1}_{i \in \pi([k])}-\frac{k}{n}\eta^2\right)=\sum_{i=1}^n a_i,
\enas
where $[k]=\{1,\ldots,k\} $ so that $\pi([k])=\{\pi(i):i=1,2, \dots, k\}$, and $a_i=\e_i^2 \mathds{1}_{i \in \pi([k])}-(k/n)\eta^2$.

To bound \eqref{etabd}, we will argue as in Lemma \ref{hoeff}. Observe that for $V$ a standard normal random variable independent of $U_k$,
\begin{equation*}
\begin{split}
\E \exp \left(\theta^2 \frac{U_k^2}{k}\right) &= \E \exp \left( \sqrt{2} \theta \frac{ V}{\sqrt{k}}U_k \right) \\
&=\E \exp \left( \sqrt{2} \theta \frac{|V| \text{sgn}(V)}{\sqrt{k}}U_k \right)\\
 &= \E \exp \left( \sqrt{2} \theta  \frac{|V|}{\sqrt{k}}U_k \Bvert \text{sgn}(V)=1\right)P(\text{sgn}(V)=1) \\
&\qquad + \E \exp\left( \sqrt{2} \theta  \frac{|V|}{\sqrt{k}} (-U_k)\Bvert \text{sgn}(V)=-1\right)P(\text{sgn}(V)=-1).
\end{split}
\end{equation*}
Now using the independence of $|V|$ and $\text{sgn}(V)$, and that $\text{sgn}(V)$ is a symmetric $\pm 1$ random variable, we obtain
\bea\label{symm}
\E \exp \left(\theta^2 \frac{U_k^2}{k}\right)=\frac{1}{2}\left[\E \exp \left( \sqrt{2} \theta U_k \frac{|V|}{\sqrt{k}}\right)+ \E \exp\left( \sqrt{2} \theta (-U_k) \frac{|V|}{\sqrt{k}}\right)\right]. 
\ena

 Recall that random variables $X_1, X_2, \dots, X_n$ are said to be negatively associated, see  \cite{joag1983negative}, if for any two disjoint index sets $I$ and $J$, 
\bea \label{fg.inc.dec}
\E[f(X_i, i\in I)g(X_j, j\in J)] \le \E[f(X_i, i\in I)]\E[g(X_j, j\in J)] 
\ena
for all coordinatewise nondecreasing functions $f: \mathbb{R}^{|I|} \to \mathbb{R}$ and $g: \mathbb{R}^{|J|} \to \mathbb{R}$. 

Let $X_1,\ldots,X_n$ be negatively associated. It is immediate that $aX_1+b,\ldots,aX_n+b$ are negatively associated for all $a \ge 0$ and $b \in \mathbb{R}$. In addition, letting $Y_i=-X_i$ for all $i=1,\ldots,n$,  for $f$ and $g$ coordinatewise nondecreasing functions and $I$ and $J$ disjoint index sets, as $-f(-\cdot)$ is coordinatewise nondecreasing, we have
\begin{multline*}
\E[f(Y_i, i\in I)g(Y_j, j\in J)] = \E[(-f(-X_i, i\in I))(-g(-X_j, j\in J))] \\
\le 
\E[(-f(-X_i, i\in I))]\E[(-g(-X_j, j\in J))] 
= \E[f(Y_i, i\in I)]\E[g(Y_j, j\in J)],
\end{multline*}
demonstrating that $-X_1,\ldots,-X_n$ are negatively associated. Combining these two facts, $aX_1+b,\ldots,aX_n+b$ are negatively associated for all $a \in \mathbb{R}$ and $b \in \mathbb{R}$.
By a direct inductive argument on \eqref{fg.inc.dec},  
\bea\label{ind}
\E\left[\prod\limits_{i=1}^n f_i (X_i)\right] \le \prod\limits_{i=1}^n \E\left[ f_i (X_i)\right]
\ena
whenever the functions $f_i, i=1,2, \dots, n$ are all nondecreasing.

By Theorem 2.11 of  \cite{joag1983negative}, taking the real numbers in Definition 2.10 there to consist of $k$ ones and $n-k$ zeros, the indicators $\mathds{1}_{1 \in \pi([k])}, \ldots, \mathds{1}_{n \in \pi([k])}$
are negatively associated; hence so are 
$a_1, \ldots, a_n$ and $-a_1, \ldots, -a_n$.
Thus, by \eqref{ind}, we have 
\begin{multline}
\label{eq:E.cond.on.V}
\E\left[ \exp\left( \sqrt{2} \theta U_k \frac{|V|}{\sqrt{k}}\right)\Bvert V\right] = \E\left[ \exp\left( \sqrt{2} \theta 
\sum_{i=1}^n a_i
\frac{|V|}{\sqrt{k}}\right)\Bvert V\right]\\
\le \prod_{i=1}^n \E\left[ \exp\left( \sqrt{2} \theta 
a_i \frac{|V|}{\sqrt{k}}\right)\Bvert V\right] 
= \prod_{i=1}^k \E\left[ \exp\left( \sqrt{2} \theta 
\left( \e_{\pi(i)}^2  -\eta^2 \right)
\frac{|V|}{\sqrt{k}}\right)\Bvert V\right].
\end{multline}

Now since $-\eta^2 \le \e_{\pi(i)}^2-\eta^2 \le B^2 - \eta^2$, using Hoeffding's lemma \eqref{Hoeffding.lemma} with $\mu=b^2-\eta^2$, the mean of $\e_{\pi(i)}^2-\eta^2$, we obtain

	\begin{multline*}
\prod_{i=1}^k \E\left[\exp\left( \sqrt{2} \theta (\e_{\pi(i)}^2 - \eta^2) \frac{|V|}{\sqrt{k}}\right)\Bvert V\right]  \le \exp\left(\frac{B^4\theta^2 V^2}{4k}+\sqrt{2}\theta \mu \frac{|V|}{\sqrt k} \right)^k\\
=\exp\left(\frac{B^4\theta^2 V^2}{4}+\sqrt{2}\theta \mu \sqrt k {|V|} \right)\\
\le \exp\left(\frac{B^4\theta^2 V^2}{4}+\sqrt{2}\theta \mu \sqrt k {V} \right) + \exp\left(\frac{B^4 \theta^2 V^2}{4} + \sqrt{2}\theta \mu \sqrt{k} (-V) \right).
\end{multline*}

Using that $V$ and $-V$ have the same distribution, taking expectation in \eqref{eq:E.cond.on.V} and then applying the non-central chi square identity \eqref{chi}
yields
\begin{multline} \label{eq:na.for.a}
\E\left[ \exp\left( \sqrt{2} \theta U_k \frac{|V|}{\sqrt{k}}\right)\right] \le 2\E\left[ \exp\left(\frac{B^4\theta^2 V^2}{4}+\sqrt{2}\theta \mu \sqrt k {V} \right)\right]\\
= \frac{2}{\sqrt{1- B^4 \theta^2/2}}\exp \left(\frac{k \theta^2 \mu^2}{\sqrt{1- B^4 \theta^2/2}} \right)
\le \frac{8}{3} \exp\left(\frac{4}{3}k \theta^2 \mu^2\right)
\end{multline}
for all $0 \le \theta \le \theta_5$, by \eqref{theta1}.

Using the fact that $-a_1,\ldots,-a_n$ are negatively associated and that $-a_i$ and $a_i$ have supports over intervals of equal length for all $i=1,2,\ldots,n$, \eqref{eq:na.for.a} holds with $U_k$ replaced by $-U_k$.
Thus, by \eqref{symm},
\bea\label{Ukbd}
\E \exp \left(\theta^2 \frac{U_k^2}{k}\right) \le \frac{8}{3} \exp\left(\frac{4}{3}k \theta^2 \mu^2\right) \qmq{for $0 \le \theta\le \theta_5$.}
\ena

Using \eqref{eq21} we see that $0 \le 4K_2 \lambda^2 \le \theta_5^2$, and as $k \le \frac{2n}{3}$, by \eqref{etabd} and \eqref{Ukbd}, and recalling that $\mu=b^2-\eta^2$, we have
\begin{multline}\label{eq:second.expectation}
\E \exp\left(4K_2 \lambda^2 k(\gamma_1^2-\eta^2)^2\right) \\
\le \frac{8}{3} \exp \left(\frac{16}{3}K_2\lambda^2 k (b^2-\eta^2)^2 \right)\le 3 \exp \left(\frac{32}{9}K_2 \lambda^2 n(b^2-\eta^2)^2 \right).
\end{multline}

For the third expectation in \eqref{eq:Cauchy.twice},
again by \eqref{eq21}, $0 \le 2 \lambda \le \theta_2.$ Hence by \eqref{eq27},
\bea \label{eq:third.expectation}
\E \exp(2\lambda T) \le \exp\left(3 + \frac{4c_1 \lambda^2a^2}{n} + 4c_2 \lambda^2 n (b^2-\eta^2)^2\right).
\ena
Applying bounds \eqref{eq:first.expectation}, \eqref{eq:second.expectation} and \eqref{eq:third.expectation} in \eqref{eq:Cauchy.twice}, and setting
\beas	
Q_{12}= 1+\frac{3K_1\lambda^2 a^2}{n}+\frac{32K_2 \lambda^2 n(b^2-\eta^2)^2}{9} \qmq{and} Q_3=3 + \frac{4c_1 \lambda^2a^2}{n} + 4c_2 \lambda^2 n (b^2-\eta^2)^2, 	
\enas	
we obtain
\begin{multline*}
\E \exp(\lambda T_L + \lambda T)
\le 3^{1/4}\exp\left(C \log k + \frac{1}{4} Q_{12}+ \frac{1}{2}Q_3 \right)\\
\le 2 \exp \left(C \log k +2 + \frac{(3 K_1  + 8c_1)\lambda^2 a^2}{4n} + \frac{(8K_2+ 18c_2)}{9}\lambda^2 n(b^2-\eta^2)^2 \right).
\end{multline*}

Again by \eqref{eq21}, $3K_1 + 8c_1=4K_1$ and $8K_2 + 18c_2 =9K_2$. Since $k \le 2n/3$, we have
\beas
\log k = \log n - \log (n/k) \le \log n - \log (3/2).
\enas
Thus, using from \eqref{eq21} that $C\log(3/2)=  \log 4 + 2$,
\beas
\E \exp(\lambda T_L + \lambda T) \le 2\exp\left( C\log n - C \log (3/2) + 2 + \frac{K_1\lambda^2 a^2}{n} + K_2\lambda^2 n (b^2 - \eta^2)^2\right)\\
=\frac{1}{2}\exp\left( C\log n + \frac{K_1\lambda^2 a^2}{n} + K_2\lambda^2  n (b^2 - \eta^2)^2\right).
\enas

In like manner we obtain this same bound on $\E \exp(\lambda T_R + \lambda T)$, so \eqref{eq30}, now yields
\begin{equation*}
\exp(\lambda \max_{1 \le i\le n} |W_i - \eta Z_i|)  \le \exp\left( C\log n + \frac{K_1\lambda^2 a^2}{n} + K_2\lambda^2  n (b^2 - \eta^2)^2\right).
\end{equation*}
This step completes the induction, and the proof. \end{proof}


\noindent \textit{Proof of Theorem \ref{exchangeable}:} Let ${\cal A}$ be the set of the $r$ distinct values $\{a_1,\ldots,a_r\}$ and let $\e_1, \e_2, \dots, \e_n$ be exchangeable random variables taking values in ${\cal A}$. Let 
\beas
{\textbf M}=(M_1,\ldots,M_r) \qmq {where for $j=1,\ldots,r$ we set} M_j = \sum_{i=1}^n {\mathds{1}}(\e_i=a_j), 
\enas
the number of components of the multiset $\pmb{\e}=\{\e_1,\ldots,\e_n\}$ that take on the value $a_j$. With ${\cal L}$ denoting distribution, or law, clearly 
\beas
{\cal L}(\e_1,\e_2,\dots, \e_n) = \sum_{{\textbf m} \ge 0}{\cal L}(\e_1,\e_2,\dots, \e_n|{\textbf M}={\textbf m})P({\textbf M}={\textbf m})
\enas
where ${\textbf m}=(m_1,\ldots,m_r)$ and ${\textbf m} \ge 0$ is to be interpreted componentwise. As ${\textbf M}$ is a symmetric function of $\e_1,\e_2,\dots, \e_n$, the conditional law ${\cal L}(\e_1,\e_2,\dots, \e_n|{\textbf M}={\textbf m})$ inherits exchangeability from ${\cal L}(\e_1,\e_2,\dots, \e_n)$, that is,
\beas
{\cal L}(\e_1,\e_2,\dots, \e_n|{\textbf M}={\textbf m})=_d {\cal L}(\e_{\pi(1)},\e_{\pi(2)},\dots, \e_{\pi(n)}|{\textbf M}={\textbf m})
\enas
where $\pi$ is uniformly chosen from ${\cal P}_n$. In particular, given ${\textbf M}={\textbf m}$,
\beas
\sum_{i=1}^k \e_i =_d \sum_{i=1}^k \e_{\pi(i)} \qmq{for all $k=1,\ldots,n$}
\enas
where $=_d$ denotes equality in distribution. Hence, \eqref{eq20} of Theorem \ref{thm4.1'} yields the version of the first claim of Theorem \ref{exchangeable} when conditioning on ${\textbf M}$, and taking expectation over ${\textbf M}$ yields that result.

We now demonstrate the second claim under the assumption that $0 \not \in \A$, which together with ${\cal A}$ finite implies that
\bea \label{eq:here.is.the.trouble}
\nu=\min_{a \in {\cal A}}|a|
\ena
is positive. With this value of $\nu$ the constants $c_1, c_2$ and $\theta_2$ as given by Theorem \ref{thm3.2'} depend only on ${\cal A}$, and let $C,K_1,K_2$ and $\lambda_0$ be as given in \eqref{eq21} for this $\nu$. As $\gamma \ge \nu$, conditional on $\e_1,\ldots,\e_n$, inequality \eqref{eq20} of Theorem \ref{thm4.1'} holds for $\eta=\gamma$, and the argument is completed by taking expectation over ${\textbf M}$ as for the proof of the first claim.

For the last claim, under the hypotheses that $\e_1,\ldots,\e_n$ are i.i.d. mean zero random variables, since $K_1$ depends only on $\A$, by Lemma \ref{hoeff} there exists $\lambda>0$ depending only on $\A$ such that
\beas
\E\left( \frac{K_1 \lambda^2 S_n^2}{n}\right) \le 2.
\enas
Thus from the second claim of the theorem we obtain
\beas
\E \exp(\lambda \max_{0\le k \le n}|W_k - \sqrt{n} \gamma  B_{k/n} |) \le 2 \exp(C \log n),
\enas
and applying Markov's inequality yields 
\begin{equation*}
\begin{split}
P\left(\max_{0\le k \le n}|W_k - \sqrt{n} \gamma  B_{k/n} | \ge \lambda^{-1} C \log n + x\right) &\le \frac{\E \exp(\lambda \max_{0\le k \le n}|W_k - \sqrt{n} \gamma  B_{k/n} |)}{\exp(C \log n)}e^{-\lambda x}\\
& \le \frac{2 \exp(C\log n)}{\exp(C \log n)}e^{-\lambda x}
= 2 e^{- \lambda x}.
\end{split}
\end{equation*} 

\qed

\section{Proof of Theorem \ref{thm3}}\label{sec4}
In this final section we prove 
Theorem \ref{thm3} by first demonstrating a `finite $n$ version' of the desired result in the following lemma.

\begin{lem} \label{lem5.1'}
There exists a constant $A$ such that for every finite set ${\cal A}$ of real numbers not containing zero, there exists a constant $\lambda > 0$ such that for any positive integer $n$, any $\e, \e_1, \e_2, \dots , \e_n$ i.i.d. random variables with mean zero and variance one satisfying $\E\e^3=0$ and taking values in $\A$, and $S_k = \sum_{i=1}^k \e_i, k=1, \dots, n$, it is possible to construct a version of the sequence $(S_k)_{0 \le k\le n}$ and Gaussian random variables $(Z_k)_{0\le k\le n}$ with mean zero and ${\rm Cov}(Z_i, Z_j) = i \wedge j$ on the same probability space such that
\bea\label{other}
\E \exp(\lambda |S_n - Z_n|) \le A
\ena
and
\bea\label{maxineq}
\E \exp (\lambda \max_{0 \le k\le n} |S_k - Z_k|) \le A \exp (A \log n).
\ena
\end{lem}
\begin{proof}
As in Theorem \ref{thm4.1'} it suffices to prove the result with the maximum taken over $1 \le k \le n$.
Recall the positive constant $\theta_1$ from Theorem \ref{thm3.1'}, the values $\vartheta_{{\ell(X)}}$ from Lemma \ref{hoeff}, $B$ from \eqref{def:B.is.max}, and let $C, K_1, K_2$ and $\lambda_0$ be as in Theorem \ref{exchangeable} for $\nu = \min_{a \in {\cal A}}|a|$. Set
\bea\label{lambdabd}
\lambda = \min\left\{\frac{\theta_1}{2}, \frac{ \lambda_0}{4}, \frac{\vartheta_{{\ell(\e)}}}{4\sqrt{K_1}},\frac{\vartheta_{{\ell(\e^2)}}}{\sqrt{2}}, \frac{1}{B+1}\right\}.
\ena

Let $g^n(s)$ and $h^n(z)$ denote the mass function of $S_n$ and the density of $Z_n$ respectively; in particular $h^n(z)$ is just the ${\cal N}(0,n)$ density. By Theorem $\ref{thm3.1'}$, as $2\lambda \le \theta_1$, with ${\cal S}^n$ the support of $S_n$,
there is a joint probability function $\psi^n(s,z)$ on ${\cal S}^n \times \mathbb{R}$ such that
\bea \label{eq:psi.marginals.in.proof}
\int_{\mathbb R} \psi^n (s, z) dz = g^n (s), \quad \sum_{s \in {\cal S}^n} \psi^n (s, z) = h^n (z),
\ena
and
\bea\label{eq31}
\int_{\mathbb R} \left[ \sum_{s \in {\cal S}^n} \exp (2\lambda |s - z|) \psi^n (s, z)\right] dz \le 8.
\ena

Given any multiset of values $\pmb{\e}=\{\e_1,\ldots,\e_n\}$ from ${\cal A}$, let $\rho_{\pmb{\e}}^n({\textbf s},{\textbf z})$ be the joint density function guaranteed by Theorem $\ref{thm4.1'}$; from that result, the marginal distributions of $\s$ and ${\textbf z}$ are, respectively,  
$f_{\pmb{\e}}^n(\s)$ as in \eqref{eq23}, and $\phi^n (\mz)$, that of a mean zero Gaussian vector with covariance \eqref{eq19}.

For any $s \in {\cal S}^n$, define
\beas
{\cal B}^n(s) = \{ \{\e_1, \e_2, \dots, \e_n\}: \sum_{i=1}^n \e_i = s \}.
\enas
Now, recalling the definition \eqref{eq:def.incrementsfrompath} of $\pmb{\e}^{\s}$,
for $s \in {\cal S}^n$, $\s$ such that ${\pmb{\e}}^\s \in {\cal B}^n(s)$, $z \in \mathbb{R}$ and ${\widetilde \mz} \in \mathbb{R}^n$, let

\bea\label{eq32}
\gamma^ n (s, z, \s, {\widetilde \mz}) = \psi^n(s, z) P(\pmb{\e}={\pmb{\e}}^{\s}|S_n=s) \rho_{{\pmb{\e}}^\s}^n(\s, {\widetilde \mz}),
\ena
where the multiset $\pmb{\e}$ on the right hand side is composed of $n$ independent random variables distributed as $\e$. 
Interpreting \eqref{eq32} in terms of a construction, to obtain $(S, Z, \mathbf{S}, \mathbf{\widetilde Z})$ one first samples the joint values $S$ and $Z$ of the coupled random walk and Gaussian path at time $n$, then conditional on the terminal value $S$, one samples increments $\pmb{\e}$ consistent with the path $\s$ from their i.i.d. distribution, and finally one couples a walk ${\textbf S}$ to the discrete Brownian bridge $\mathbf{\widetilde Z}$ in such a way that a certain multiple of it and $(W_1,\ldots,W_n)$ given by 
\bea \label{def:Wi.in.proof}
W_i=S_i-\frac{i}{n}S_n
\ena
are close.

To verify that \eqref{eq32} determines a probability function, recalling \eqref{eq22}, note first that
\begin{multline*}
\sum_{\s: \e^{\s} \in {\cal B}^n(s)} P(\pmb{\e}={\pmb{\e}}^{\s}|S_n=s) \rho_{{\pmb{\e}}^\s}^n(\s, {\widetilde \mz}) \\= \sum_{\pmb{\delta} \in {\cal B}^n(s)} \sum_{\s \in {\cal A}_{\pmb{\delta}}^n}P(\pmb{\e}={\pmb{\delta}}|S_n=s) \rho_{{\pmb{\delta}}}^n(\s, {\widetilde \mz})
= \sum_{\pmb{\delta} \in {\cal B}^n(s)} P(\pmb{\e}={\pmb{\delta}}|S_n=s)
\sum_{\s \in {\cal A}_{\pmb{\delta}}^n} \rho_{{\pmb{\delta}}}^n(\s, {\widetilde \mz}) \\= \sum_{\pmb{\delta} \in {\cal B}^n(s)} P(\pmb{\e}={\pmb{\delta}}|S_n=s)
\phi^n ({\widetilde \mz}) = \phi^n ({\widetilde \mz}).
\end{multline*}
Now by \eqref{eq:psi.marginals.in.proof}, 
\bea \label{eq:marginals.z.bfz}
\sum_{s \in {\cal S}^n} \sum_{\s: \e^{\s} \in {\cal B}^n(s)}\gamma^ n (s, z, \s, {\widetilde \mz}) = h^n(z) \phi^n (\widetilde \mz),
\ena
and integrating over $z$ and ${\widetilde \mz}$ yields 1.

Let $(S, Z, \mathbf{S}, \mathbf{\widetilde Z})$ be a random vector sampled from $\gamma^n(s,z,\s,{\widetilde \mz})$, and 
define $\mathbf{Z} = (Z_1,\dots,Z_n)$ by
\beas
Z_i = \widetilde Z_i + \frac{i}{n} Z.
\enas
Using that $Z$ and $\mathbf{\widetilde Z}$ are independent by \eqref{eq:marginals.z.bfz}, and that the latter has covariance given by \eqref{eq19}, it follows that $\mathbf{Z}$ is a mean zero Gaussian random vector with ${\rm Cov}(Z_i, Z_j) = i \wedge j$. 

Regarding the marginals of $\s$,
integrating \eqref{eq32} over $z$ and ${\widetilde \mz}$, with $f_{\pmb{\e}}^n(\s)$ given by \eqref{eq23}, we obtain
\beas
\int_{\mathbb{R}^n} \int_{\mathbb{R}} \gamma^ n (s, z, \s, {\widetilde \mz}) dz d{\widetilde \mz} = g^n(s)P(\pmb{\e}={\pmb{\e}}^{\s}|S_n=s)f_{\pmb{\e^{\s}}}^n(\s) = P(\pmb{\e}={\pmb{\e}}^{\s})f_{\pmb{\e^{\s}}}^n(\s) = P(\pmb{\e}={\pmb{\e}}^{\s})\frac{1}{|{\cal A}^n_{\pmb{\e}^{\s}}|}.
\enas
The first term is the likelihood that the independently generated increments corresponding to those of $\s$, while the second term is the chance that these increments will be arranged by the uniform permutation in an order that produces $\s$. Hence, the marginal correspond to the distribution of $\mathbf{S}$.

It only remains to show that the pair $(\mathbf{S}, \mathbf{Z})$ satisfies the bounds \eqref{other} and \eqref{maxineq}.
Note that for $1 \le i \le n$, recalling \eqref{def:Wi.in.proof}, we have
\bea
|S_i - Z_i| &=& \Bvert W_i+\frac{i}{n}S - \left( \widetilde Z_i + \frac{i}{n} Z \right) \Bvert \nonumber\\
&\le& |W_i - \widetilde Z_i| + \frac{i}{n} |S - Z|. \label{eq:SiYibd}
\ena
From \eqref{eq32}, one can easily check 
that the conditional distribution of $(\mathbf{S}, \mathbf{\widetilde Z})$ given $(\pmb{\e}^{\mathbf{S}}, Z) = (\pmb{\e}, z)$ is $\rho_{\pmb{\e}}^n({\textbf s},{\tilde \mz})$.

Let ${\gamma}^2=n^{-1}\sum_{i=1}^n \e_i^2$ and recall $\nu =\min_{a \in {\cal A}}|a| >0$. As $\gamma \ge \nu$ and $4 \lambda \le \lambda_0$ 
 by \eqref{lambdabd}, we may invoke 
Theorem \ref{thm4.1'} conditional on $\{\pmb{\e},Z\}$, and choosing
$\eta=\gamma$ we obtain
\bea \label{eq:applyTheorem4.1}
\E (\exp (4 \lambda \max_{1 \le i \le n} |W_i - \gamma \widetilde Z_i|)\big| \pmb{\e}, Z)\le \exp \left( C \log n + \frac{16 K_1 \lambda^2 S_n^2}{n} \right),
\ena
with $C$ and $K_1$ depending only on $\A$. Applying the Cauchy-Schwarz inequality and \eqref{eq31}, as $S$ and $Z$ are measurable with respect to $\{\pmb{\e}, Z\}$, from \eqref{eq:SiYibd} we obtain
\begin{align}
&\E \exp(\lambda \max_{1 \le i \le n} |S_i - Z_i|)\nonumber \\
 \le &\Big[\E \Big( \E\big(\exp (\lambda \max_{1 \le i \le n} |W_i - \widetilde Z_i|)\big| \pmb{\e}, Z \big) \Big){^2} \E \exp(2\lambda |S - Z|)\Big]^{1/2}\nonumber\\
 \le & \Big[8 \E \Big( \E\big(\exp (\lambda \max_{1 \le i \le n} |W_i - \widetilde Z_i|)\big| \pmb{\e}, Z \big) \Big)^{2} \Big]^{1/2}.\label{eq:with.kappa}
\end{align}
Using conditional Jensen's inequality, the triangle inequality and the convexity of the exponential function in the first three lines below, \eqref{eq:applyTheorem4.1} yields
 \begin{align}
&\left( \E\big(\exp (\lambda \max_{1 \le i \le n} |W_i - \widetilde Z_i|)\big| \pmb{\e}, Z \big) \right)^2\nonumber\\
&\le \E\big(\exp (2\lambda \max_{1 \le i \le n} |W_i -\widetilde Z_i|)\big| \pmb{\e}, Z \big) \nonumber\\
&\le \frac{1}{2}\E\big(\exp (4\lambda \max_{1 \le i \le n} |W_i - \gamma \widetilde Z_i|)\big| \pmb{\e}, Z \big) +\frac{1}{2}\E\big(\exp (4\lambda \max_{1 \le i \le n} |\gamma \widetilde Z_i - \widetilde Z_i|)\big| \pmb{\e}, Z \big) \nonumber\\
&\le \frac{1}{2}\exp \left( C \log n + \frac{16K_1 \lambda^2 S_n^2}{n} \right) +\frac{1}{2}\E\big(\exp (4\lambda |\gamma-1|  \max_{1 \le i \le n} |\widetilde Z_i|)\big| \pmb{\e}, Z \big) \nonumber \\ 
& \le \exp(C \log n) + \frac{1}{2}\E\big(\exp (4\lambda |\gamma-1|  \max_{1 \le i \le n} |\widetilde Z_i|)\big| \pmb{\e}, Z \big).\label{eq33}
\end{align}
For the first term in the fourth line, Lemma \ref{hoeff} yields
\beas 
\E \exp  \left(\frac{16K_1 \lambda^2 S_n^2}{n} \right) \le 2,
\enas
since $\e_1$ has mean zero, $|\e_1|\le B$ in \eqref{def:B.is.max} and $4 \sqrt{K_1} \lambda \le \vartheta_{\ell(\e)}$ by \eqref{lambdabd}. 

For the second term in \eqref{eq33}, observe that conditional on  $(\pmb{\e}, Z)$, 
$\mathbf{\widetilde Z}$ is a mean zero multivariate Gaussian random vector with covariance given by \eqref{eq19}. Equivalently, conditional on  $(\pmb{\e}, Z)$, the distribution of $(\widetilde Z_i/\sqrt n)_{1 \le i \le n}$ is that of a Brownian bridge on $[0,1]$ sampled at times 
$1/n, 2/n, \dots, 1$. Thus, letting $B_t, t \in [0,1]$ be a Brownian bridge independent of $(\pmb{\e}, Z)$, since $\gamma$ is a function of $\pmb{\e}$, we have
\begin{equation*}
\begin{split}
\E&\big(\exp (4\lambda  |\gamma - 1| \max_{1 \le i \le n} |\widetilde Z_i|)\big| \pmb{\e}, Z \big)\\
= &\E\big(\exp (4\sqrt n \lambda |\gamma-1|\max_{1 \le i \le n}\frac{ |\widetilde Z_i|}{\sqrt n}) \big| \pmb{\e}, Z \big)\\
= &\E\big(\exp (4\sqrt n \lambda |\gamma - 1| \max_{t \in [n]/n} |B_t|) \big| \pmb{\e}, Z \big) \\
\le &\E\big(\exp (4\sqrt n \lambda |\gamma - 1| \max_{0 \le t \le 1} |B_t|) \big| \pmb{\e}, Z \big) \\
  \le &\E\big(\exp (4\sqrt n \lambda |\gamma  - 1| \max_{0\le t \le 1} B_t) + \exp (4\sqrt n \lambda |\gamma - 1| \max_{0\le t \le 1} (-B_t)) \big| \pmb{\e}, Z\big).
\end{split}
\end{equation*}
From \cite{smirnoff1939ecarts}, the distribution of $X=\max_{0\le t \le 1} B_t$ is given by
\beas
P(X\le x)= 1- \exp(-2x^2) \qmq{for} x\ge 0.
\enas
Using this identity, and the fact that $-B_t$ is also a Brownian bridge, it is straightforward to show that for any real number $a$, we have
\beas
\E\big(\exp (a \max_{0\le t \le 1} B_t) + \exp (a \max_{0\le t \le 1} (-B_t))  \big)\le 2+\sqrt{2\pi} a\exp(a^2/8).
\enas
Thus, since $B_t$ and $\gamma$ are respectively independent of, and a function of, $\pmb{\e}$, we obtain
\begin{align} 
\E&\big(\exp (4\lambda |\gamma-1|\max_{1 \le i \le n} |\widetilde Z_i|)\big| \pmb{\e}, Z \big)\nonumber\\
\le & 2+\sqrt{2\pi}4 \sqrt n \lambda |\gamma - 1| \exp \left(2\lambda^2 n(\gamma-1)^2\right)\nonumber\\
\le & 2 + 4 (B + 1)  \sqrt {2\pi n} \lambda \exp \left(2\lambda^2n(\gamma^2-1)^2\right)\label{eq:take.exp.over}
\end{align}
where in the last step, we used $|\gamma-1|\le B+1$ where $B$ is given by \eqref{def:B.is.max}, and that  
$\gamma \ge 0$ implies $1 \le {(\gamma +1)}^2$.\\\\
Since $\E \e_1^2=1$, we have $n(\gamma^2-1)^2= \big(\sum_{i=1}^n (\e_i^2 - \E \e_i^2)\big)^2/n$ and $\E(\e_i^2 - \E \e_i^2)=0$. As $\e^2 \le B^2$ and $0\le\sqrt{2}\lambda \le \vartheta_{{\ell(\e^2)}}$, by \eqref{lambdabd}, Lemma \ref{hoeff} yields
\beas
\E \exp \left(2\lambda^2 n(\gamma^2-1)^2\right)\le 2.
\enas
Additionally, since $\lambda (B+1)\le 1$ by \eqref{lambdabd}, taking expectation in \eqref{eq:take.exp.over} yields 
\bea\label{eq34}
\E(\exp (4\lambda |\gamma-1| \max_{1 \le i \le n} |\widetilde Z_i|))
=2 + 8(B+1)  \sqrt {2\pi n} \lambda \le 
\exp (C_1 \log n)
\ena
for some universal constant $C_1$.

Thus, by \eqref{eq:with.kappa}, \eqref{eq33} and \eqref{eq34},  we have
\begin{equation*}
\begin{split}
\E & \exp(\lambda \max_{1 \le i \le n} |S_i - Z_i|)\\
 \le &  \Big[8 \E \Big( \exp \left( C \log n \right) +\frac{1}{2}\E\big(\exp (4\lambda |\gamma-1|\max_{1 \le i \le n} |\widetilde Z_i|)\big| \pmb{\e}, Z \big) \Big] \Big) \Big]^{1/2}\\
 \le & 8^{1/2}\Big[ \exp ( C \log n) +\frac{1}{2}\exp (C_1 \log n) \Big]^{1/2}\\
 \le & A \exp(A \log n)
\end{split}
\end{equation*}
for some universal constant $A$, which we may take to be at least $8$. The proof of $\eqref{maxineq}$ is now complete. Lastly note that $\widetilde Z_n=0$ implies $Z_n=Z$, hence \eqref{eq31} yields \eqref{other} as $A \ge 8$.
\end{proof}

Theorem $\ref{thm3}$ follows from Lemma \ref{lem5.1'} in exactly the same way as Theorem 1.5 follows from Lemma 5.1 in \cite{chatterjee2012new}, noting that the reasoning applied at this step does not depend on the support of the summand variables of the random walk.\\

\section*{Acknowledgement} The authors would like to thank Sourav Chatterjee for bringing the key issues in his work \cite{chatterjee2012new} to our attention, and for many helpful discussions. The second author was partially supported by NSA grant H98230-15-1-0250.

\bibliographystyle{amsplain}

\providecommand{\bysame}{\leavevmode\hbox to3em{\hrulefill}\thinspace}
\providecommand{\MR}{\relax\ifhmode\unskip\space\fi MR }
\providecommand{\MRhref}[2]{%
  \href{http://www.ams.org/mathscinet-getitem?mr=#1}{#2}
}
\providecommand{\href}[2]{#2}

\end{document}